\documentclass[11pt,leqno]{article}

\usepackage{amsthm,amsfonts,amssymb,amsmath,oldgerm}
\usepackage{fullpage}
\usepackage{graphicx}
\usepackage{float}

\usepackage{mathrsfs}
\usepackage{xcolor}

\numberwithin{equation}{section}

\newcommand{\LT}{L^2({\mathbb{R})}}

\def\eps {\varepsilon}

\newcommand{\R}{\mathbb R}

\newcommand{\eq}[2]{\begin{equation}\begin{split}#1\end{split}\label{#2}\end{equation}}

\newcommand{\eqnn}[1]{\begin{equation}\begin{split}#1\end{split}\nonumber\end{equation}}

\newcommand{\RM}{{\mathbb{R}}}
\newcommand{\CM}{{\mathbb{C}}}

\newcommand{\LM}{{\mathcal{L}}}
\newtheorem{theorem}{Theorem}[section]
\newtheorem{proposition}[theorem]{Proposition}
\newtheorem{corollary}[theorem]{Corollary}
\newtheorem{lemma}[theorem]{Lemma}
\newtheorem{remark}[theorem]{Remark}
\theoremstyle{definition}

\newtheorem{assumption}[theorem]{Assumption}

\allowdisplaybreaks[3]

\title{Orbital Stability of Smooth Solitary Waves for the Novikov Equation}

\author{Brett Ehrman\thanks{Department of Mathematics, University of Kansas, 1460 Jayhawk Boulevard, 
Lawrence, KS 66045, USA; ehrman.brett@ku.edu}~~
Mathew~A.~Johnson\thanks{Department of Mathematics, University of Kansas, 1460 Jayhawk Boulevard, 
Lawrence, KS 66045, USA; matjohn@ku.edu}~~ 
St\'ephane Lafortune\thanks{Department of Mathematics, College of Charleston, Charleston, SC 29401, USA; lafortunes@cofc.edu}
}

\date{\today}

\begin{document}

\maketitle

\begin{abstract}
We study the orbital stability of smooth solitary wave solutions of the Novikov equation, which is a Camassa-Holm type 
equation with cubic nonlinearities.  These solitary waves are shown to exist as a one-parameter family (up to spatial translations)
parameterized by their asymptotic endstate, and are encoded as critical points of a  particular action functional. As an important
step in our analysis we must study the spectrum the Hessian of this action functional, which turns out to be a nonlocal integro-differential operator
acting on $L^2(\RM)$.  We provide a combination of analytical and numerical evidence that the necessary spectral hypotheses always
holds for the Novikov equation.  Together with a detailed study of the associated Vakhitov-Kolokolov condition, our analysis indicates 
that all smooth solitary wave solutions of the Novikov equation are nonlinearly orbitally stable.
\end{abstract}


\section{Introduction}

In this paper, we consider the nonlinear stability of smooth solitary wave solutions of the Novikov equation
\begin{equation}\label{e:nov2}
u_t-u_{xxt}=u^2u_{xxx}-4u^2u_x+3uu_xu_{xx},
\end{equation}
which was originally proposed by Novikov \cite{Novikov} as part of a classification of generalized Camassa-Holm-type equations
that possess an infinite hierarchy of higher order symmetries.  The equation \eqref{e:nov2} is known to be completely integrable via the 
Inverse Scattering Transform, has infinitely many symmetries and conserved quantities, and is bi-Hamiltonian \cite{HoneNovikovH,Novikov}.
Further, the Novikov equation has been shown to model the propagation of shallow water waves of moderately large amplitude \cite{Chen2022}.
In this regard, \eqref{e:nov2} can be regarded as generalization of the Camassa-Holm equation \cite{Cam1,Cam2}
\begin{equation}
\label{CH}
u_t-u_{txx} = uu_{xxx}-3uu_x+2u_xu_{xx},
\end{equation}
 that accounts for cubic nonlinearities.

The Novikov equation \eqref{e:nov2} is known to admit a variety of both smooth and non-smooth traveling wave structures.  
Indeed, similar to the Camassa-Holm equation, equation \eqref{e:nov2} is known to admit both peaked and multi-peaked soliton solutions, known
as peakons.  The stability of peakon solutions of \eqref{e:nov2} has received considerable attention, 
and their existence, spectral and linear (in)stability, and nonlinear orbital stability has been studied in various other works: see \cite{Chen2021,Lafortune2024,Palacio2020,Palacios2021}
and references therein.  Additionally, the Novikov equation admits smooth soliton and multi-soliton solutions, as well as smooth spatially periodic traveling waves
\cite{Matsuno,EhJL2}.  Unlike their peaked counterparts, however, the dynamical stability of smooth solutions so far has received little to no attention.

In this work, we aim to study the stability of smooth solitary wave solutions of the Novikov equation \eqref{e:nov2}.  As we will see, such waves
exist as a one-parameter family (up to spatial translations) of smooth solutions that can be smoothly parameterized by their necessarily non-zero asymptotic
end state at spatial infinity.  In this setting, the well-posedness of \eqref{e:nov2} has been studied previously in \cite{Himonas2012,Ni2011}.  Further,
in \cite{Wu2011,Wu2012} it was shown that solutions to \eqref{e:nov2} with initial data $u_0\in H^s(\RM)$ with $s>\frac{3}{2}$ satisfying $u_0-u_0''>0$ exist globally in time.  
Finally, we note that the phenomena of wave-breaking was studied in the context of the Novikov equation in \cite{Chen2016,Jiang2012,Wu2012}.

The basic approach taken in this work is by now classical, essentially being an application of the methodology formalized by Grillakis, Shatah and Strauss in \cite{GSS}
for the stability of nonlinear solitary waves in Hamiltonian systems.  To this end, we we note that the Hamiltonian structure associated to \eqref{e:nov2} is 
expressed in terms of the so-called momentum density variable $m=u-u_{xx}$.  In terms of the variable $m$, the Novikov equation can be rewritten
as
\begin{equation}\label{e:nov1}
m_t+u^2m_x+3uu_xm=0.
\end{equation}
Once the existence of smooth solitary waves to \eqref{e:nov2}, or equivalently \eqref{e:nov1}, is established we will show directly
that such waves arise as critical points of an appropriate action functional constructed as a linear combination of conserved quantities
for the Novikov equation.  Following the general methodology in \cite{GSS} we then aim to determine conditions that guarantee that a given
solitary wave is a constrained local extrema of the Hessian of the associated action functional.  This requires a detailed study of the spectral 
properties of the Hessian operator, denoted in our work as $\mathcal{L}$, acting on $L^2(\RM)$ which, in the present case, is complicated by the fact that 
$\mathcal{L}$ itself is a nonlocal integro-differential operator of second-order\footnote{Equivalently and equally as difficult, the spectral problem for $\mathcal{L}$ can be recast
as a generalized eigenvalue problem for a fourth-order differential operator.}.  Through a combination of analytical and well-conditioned numerical Evans function techniques,
we provide strong evidence that the necessary spectral hypotheses on $\mathcal{L}$ in fact hold for all smooth solitary wave solutions of \eqref{e:nov1}.
With these spectral properties established, we then show that such waves are orbitally stable provided that a so-called Vakhitov-Kolokolov condition
is satisfied.  By exploiting various scaling symmetries of the Novikov equation, we provide a detailed analysis of the Vakhitov-Kolokolov condition 
in the present case and show it is equivalent to the positivity of a function which is directly numerically computable.  In particular,
our analysis shows that the Vakhitov-Kolokolov condition always holds in the present case, indicating that all smooth solitary wave solutions
of the Novikov equation \eqref{e:nov1} are nonlinearly orbitally stable.

\

The outline of our paper is as follows.   In Section \ref{S:Ex} study the existence and basic properties of smooth solitary wave solutions of the Novikov
equation \eqref{e:nov2}.  In particular, in Section \ref{S:exist}  we use phase plane analysis to establish the existence of smooth solitary wave solutions
of \eqref{e:nov2}.  We then show in Section \ref{S:cp} that such waves can be considered as critical points of an explicit action functional $\Lambda$,
which is the foundation of our stability analysis.  In Section \ref{S:SpecL} we consider the spectrum of the operator $\mathcal{L}=\frac{\delta^2\Lambda}{\delta^2 m}$
evaluated at a smooth solitary wave, considered as an (unconstrained) operator acting on $L^2(\RM)$.  The goal of this section
is to provide analytical results and strong numerical support for Assumption \ref{H1}, which effectively states that $\mathcal{L}$ acting on $L^2(\RM)$ has a simple eigenavlue
at the origin and a single negative eigenavlue.  Specifically, in Section \ref{51} we determine the essential spectrum of $\mathcal{L}$ analytically and determine
a lower bound on the possible negative point spectrum for the operator $\mathcal{L}$.  Equipped with these results, in Section \ref{52} we use numerical
Evans function techniques to present an  investigation of the point spectrum of $\mathcal{L}$.  The results of this numerical investigation is
that the spectrum of $\mathcal{L}$ indeed satisfies Assumption \ref{H1} for all smooth solitary waves of \eqref{e:nov2}.  With Assumption \ref{H1} in hand,
our main stability analysis is contained in Section \ref{S:Orbital}, culminating in our main result, Theorem \ref{T:main}, establishing orbital
stability of smooth solitary wave solutions of \eqref{e:nov2} provided a particular Vakhitov-Kolokolov condition holds.  We conclude
our study in Section \ref{S:VK} with a detailed study of the derived Vakhitov-Kolokolov condition, showing in particular that it holds for all smooth
solitary waves of \eqref{e:nov2}.

\

\noindent
{\bf Acknowledgments:}  The work of BE and MAJ was partially funded by the NSF under grant DMS-2108749.  
The work of SL was supported by a Collaboration Grant for Mathematicians from the Simons Foundation (award \# 420847).
The authors would like to thank Blake Barker for helpful conversations regarding the use  STABLAB \cite{StabLab} in our numerical Evans function investigation.
We also thank Ming Chen, Alex Himonas, and Dmitry Pelinovsky for helpful discussions.

%
%

\section{Smooth Solitary Wave Solutions}\label{S:Ex}

In this section, we begin by using phase plane analysis to prove the existence of a one-parameter family of (even) smooth
solitary wave solutions, which can be parameterized by the asymptotic value at spatial infinity.  With this in hand, we will
then demonstrate that these waves can be considered as critical points of a linear combination of conserved quantities
associated to \eqref{e:nov2}.

\subsection{Existence of Smooth Solitary Waves}\label{S:exist}

We seek traveling wave solutions of \eqref{e:nov2} of the form $u(x,t)=\phi(x-ct)$, which will correspond to stationary solutions of the PDE
\[
u_t-u_{xxt}-c(u_x-u_{xxx})+4u^2u_x=u^2u_{xxx}+3uu_xu_{xx}.
\]
After some rearranging, we see that such stationary solutions necessarily satisfy the ODE
\[
(\phi^2-c)\left(\phi-\phi''\right)'+3\phi\phi'\left(\phi-\phi''\right)=0.
\]
By elementary ODE theory, it follows that $\phi\in C^\infty(\RM)$ provided that either $\phi^2(x)<c$ or $\phi^2(x)>c$ for all $x\in\RM$.
Throughout our work, we will assume that
\begin{equation}\label{e:cond1}
\phi^2(x)<c~~{\rm for~all}~x\in\RM.
\end{equation}
Equipped with this condition, we note that multiplying by the integrating factor $(\phi-\phi'')^{-1/3}$ that the above ODE can be written as
\[
\frac{d}{dx}\left((\phi-\phi'')^{2/3}\left(c-\phi^2\right)\right)=0.
\]
By further enforcing the condition that
\begin{equation}\label{e:cond2}
\phi-\phi''>0~~{\rm for~all}~x\in\RM
\end{equation}
the above can be further reduced to the ODE
\begin{equation}\label{e:profile1}
\phi-\phi''=\frac{a}{(c-\phi^2)^{3/2}}
\end{equation}
where here $a>0$ is a constant of integration\footnote{Existence in the case when the conditions \eqref{e:cond1} or \eqref{e:cond2} are not met is discussed in Remark \ref{R:exist} below.}.  Integrating once more, we see that the above may be further reduced to quadrature to
\begin{equation}\label{e:profile_quad}
\frac{1}{2}(\phi')^2 = E+\frac{1}{2}\phi^2 - \frac{a\phi}{c\sqrt{c-\phi^2}}.
\end{equation}

\begin{figure}[t!]
\begin{center}
\includegraphics[scale=0.7]{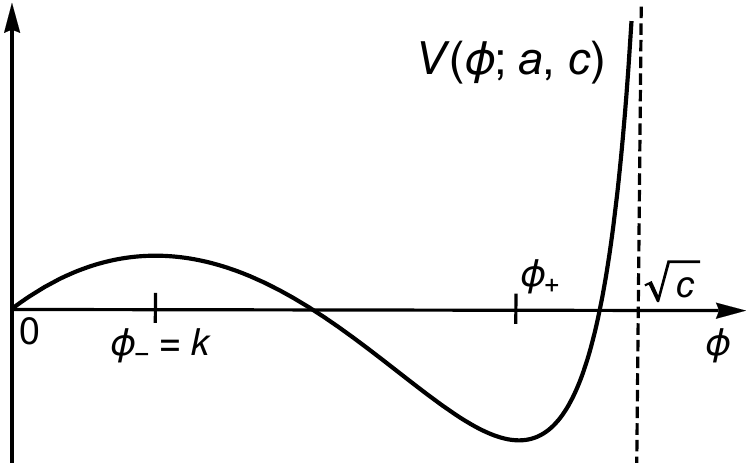}
\caption{A schematic  drawing of the effective potential $V$ for a given $c>0$ and an admissible value of $a\in\left(0,\frac{3\sqrt{3}c^2}{16}\right)$.
Note that the parameter $a$ is defined as a function of $k$ in \eqref{kcond} below.}
\end{center}
\end{figure}

It follows that the existence of solitary wave solutions of \eqref{e:nov2} can be determined by studying the potential function
\[
V(\phi;a,c):=-\frac{1}{2}\phi^2+ \frac{a\phi}{c\sqrt{c-\phi^2}}.
\]
Indeed,  it follows that smooth solitary wave solutions with profile $\phi\in C^\infty(\RM)$ satisfying the conditions \eqref{e:cond1}-\eqref{e:cond2}
correspond to homoclinic orbits of \eqref{e:profile_quad}.  To guarantee the existence of such a homoclinic orbit, we need to guarantee 
that the potential $V(\cdot;a,c)$ has both a local maximum and local minimum in the interval $\phi\in(-\sqrt{c},\sqrt{c})$.  To this end,
observe that $V'(\phi;a,c)=0$ if and only if
\begin{equation}\label{uzk}
D(\phi;a,c)=a-\phi(c-\phi^2)^{3/2}=0
\end{equation}
which, since $a>0$ by above, implies that such critical points can only occur when $\phi>0$.  Further, we note that
\[
D(0;a,c)=D(\sqrt{c};a,c)=a>0
\]
for all $a,c>0$, while the only root of $D'(\cdot;a,c)=0$ on $(0,\sqrt{c})$, corresponding to an inflection point of $V$, occurs at $\phi=\sqrt{c}/2$.
Observing that 
\[
D\left(\sqrt{c}/2\right)=a-\frac{3\sqrt{3}c^2}{16},
\]
it follows by choosing $a,c>0$ to ensure that $D(\sqrt{c}/2)<0$ that on the domain $(-\sqrt{c},\sqrt{c})$ the effective potential $V(\cdot;a,c)$ will have exactly 
two critical points 
\[
0<\phi_-<\frac{\sqrt{c}}{2}<\phi_+<\sqrt{c}
\]
with $\phi_-$ being a local maximum and $\phi_+$ being a local minimum of $V$.  These local extrema yield, respectively, a saddle point $(\phi_-,0)$
and a nonlinear center $(\phi_+,0)$ for the second-order profile equation \eqref{e:profile1}.  By the above analysis, it follows that for each $c>0$ and each such
$a$ there exists a homoclinic orbit connecting the saddle point $(\phi_-,0)$, which corresponds to a smooth solitary wave solution of \eqref{e:nov2}.
Note by translation invariance that for each such $c>0$ and $a\in(0,\frac{3\sqrt{3}c^2}{16})$ the solitary wave profile satisfying $\phi'(0)=0$
is uniquely defined.

Taken together, the above considerations establish the following result.

\begin{theorem}
\label{exth}
For each fixed $c>0$ and  $0<a<\frac{3\sqrt{3}c^2}{16}$ there exists a unique, smooth solitary wave solution $u(x,t)=\phi(x-ct)$ 
of \eqref{e:nov2} which satisfies the conditions 
\[
\phi(x)^2<c~~{\rm and}~~\phi(x)-\phi''(x)>0~~{\rm for ~all}~~x\in\RM,
\]
along with $\phi'(0)=0$.  Further, we note that $\phi$ is even and monotonically decreasing on $(0,\infty)$.
\end{theorem}

\begin{remark}
Associated with the smooth, even solitary waves $\phi$ determined in Theorem \ref{exth} we define their momentum densities as $\mu=\phi-\phi''$.  By construction,
we note that such $\mu$ are smooth, even functions that satisfy $\mu'(0)=0$ and $\mu''(0)<0$.  Further, we note that
\[
\lim_{x\to\pm\infty}\mu(x) = \lim_{x\to\pm\infty}\phi(x).
\]
Throughout our work, we will interchangeably refer to both $\phi$ and $\mu$ as being smooth solitary wave solutions of \eqref{e:nov2}.
\end{remark}

\begin{remark}\label{R:exist}
The above analysis culminating in Theorem \ref{exth} considers the case when the wave profile $\phi$ satisfies both conditions \eqref{e:cond1} and \eqref{e:cond2}.
In the case when $\phi^2<c$ and $m<0$, smooth solitary waves are shown to exist by a similar phase plane analysis in  \cite{ZXO}.
While it would be interesting to extend our analysis to this case as well, we do not pursue it here.
Finally, we note that \cite{PL} shows that no smooth solitary waves exist when $\phi^2>c$.
\end{remark}

\

Before moving on, we note that the solitary wave solutions constructed above generally exist on non-zero constant backgrounds.  Indeed, since they are constructed
as homoclinic orbits associated to the second-order ODE \eqref{e:profile1}, we see for a given $c>0$ that  the solitary waves can be 
parameterized by the single parameter $k:=\phi_1(a)$.  Note, in particular, that for each $k$ we have
\[
\lim_{x\to\pm\infty}\phi(x;k)=k,~~{\rm and}~~\lim_{x\to\pm\infty}\partial_x^\ell\phi(x;k)=0~~{\rm for~all}~\ell\geq 1
\]
and, in particular, this convergence occurs at an exponential rate.
It follows that for a fixed wavespeed $c>0$
the parameters $a$ and $E$ associated to a given solitary wave $\phi(x;k)$ can be defined in terms of the asymptotic value $k$ via
\begin{equation}\label{kcond}
a=k(c-k^2)^{3/2},\;\;E=\frac{k^2(c-2k^2)}{2c},
\end{equation}
where we used \eqref{e:profile_quad} and the expression for $a$ above to obtain the second equation.
Recalling \eqref{uzk}, we further note that for a given $a\in\RM$ the corresponding asymptotic state $k$ is the smallest of the two solutions 
to the first equation of \eqref{kcond} on the interval $(0,c)$.
Given that the maximum value of the function $a(k)$ defined in \eqref{kcond} occurs at $k_{max}=\sqrt{c}/2$, it follows that the asymptotic value of $\phi(\cdot;k)$ 
necessarily satisfies
\begin{equation}\label{kcond2}
0<k < \frac{\sqrt{c}}{2}.
\end{equation}
Taken together, this gives the following result.

\begin{corollary}
For a fixed $c>0$ the smooth solitary waves $\phi$ constructed in Theorem \ref{exth} can be smoothly parameterized by the 
asymptotic end state $k$ satisfying \eqref{kcond2}.  
\end{corollary}

Throughout most of our work, we will consider the solitary waves $\phi$, as well as their associated momentum densities $\mu=\phi-\phi''$ as smooth functions
of the parameter $k$.  When appropriate, we will use the notation $\phi(\cdot;k)$ and $\mu(\cdot;k)$ to denote this explicit dependence.

\begin{remark}
For a fixed $c>0$, we note that in the limit $k\to 0^+$ the smooth solitary wave $\phi(\cdot;k)$ limits to the one-peakon solution
\[
\phi_{\rm peakon}(x) = \sqrt{c}e^{-|x|}.
\]
The stability of the one-peakon solution of Novikov has been studied previously in \cite{Chen2021,Lafortune2024,Palacio2020,Palacios2021}, where they were shown to be spectrally, linearly, and orbitally stable to perturbations in $H^1(\RM)$,  
Additionally, they were shown in \cite{Chen2019,Lafortune2024} to be both spectrally and linearly unstable to perturbations in $W^{1,\infty}(\RM)$.
Further, in the limit $k\to(\sqrt{c}/2)^-$ the smooth solitary wave $\phi(\cdot;k)$ limits to the constant solution $\phi_{\rm cst}(x)=\frac{\sqrt{c}}{2}$.  In the present work,
we are interested in the stability of the smooth, non-constant solitary wave solutions of \eqref{e:nov2} only.
\end{remark}

Throughout the forthcoming work, we consider the wave speed $c>0$ as being fixed and our solitary waves $\phi(x;k)$ and their corresponding momentum densities $\mu(x;k)$ as being
smoothly parameterized by the asymptotic value $k$ satisfying \eqref{kcond2}.

\subsection{Smooth Solitary Waves as Critical Points}\label{S:cp}

Let $c>0$ be fixed.  In the previous section, we used elementary phase plane analysis to prove the existence of a one-parameter family of smooth, even
solitary wave solutions $\phi(x;k)$ of \eqref{e:profile1}.  In this section, we show that these solutions correspond to critical points of a particular linear combination
of conserved quantities for the Novikov equation \eqref{e:nov2}.

To this end, we note that the Novikov equation is known to formally have the following conserved quantities:
\[
\widetilde{\mathcal{E}}(m) := \int_\RM um~dx~~\widetilde{F}_1(m) := \int_{\RM} m^{2/3}dx~~\widetilde{F}_2(m) := \int_{\RM}\left(m^{-8/3}m_x^2+9m^{-2/3}\right)dx,
\]
where we note $\widetilde{\mathcal{E}}$ may be considered well-defined as a function of $m$, instead of both $m$ and $u$, since $u=(1-\partial_x^2)^{-1}m$.  
Of course, the above functionals are only well-defined for solutions 
on a zero background.  When solutions of \eqref{e:nov2} are considered on the non-zero constant background state $m(x,t)\to k$ as $x\to\pm\infty$, the conserved
quantities must be redefined.  For a fixed $k>0$, we define the set
\[
X_k:=\left\{m-k\in H^1(\RM):m(x)>0~{\rm for~all}~x\in\RM\right\}
\]
and redefine the conserved quantities for $m\in X_k$ as
\eqnn{
\mathcal{E}(m):=\int_\RM\left(m(1-\partial_x^2)^{-1}m-k^2\right)dx,~~F_1(m):=\int_\RM\left(m^{2/3}-k^{2/3}\right)dx
}
and
\eqnn{
F_2(m):=\int_\RM\left(m^{-8/3}m_x^2+9\left(m^{-2/3}-k^{-2/3}\right)\right)dx.
}

\begin{remark}\label{R:wellposed}
We note that local well-posedness of the Novikov equation \eqref{e:nov2} is known for initial data $u(0,\cdot)\in H^s(\RM)$ with $s>\frac{3}{2}$: see, for example,
\cite{Ni2011, Himonas2012}. Moreover, it is known that the local solutions exist for all time
and do not break assuming that the associated momentum density $m$ is initially strictly positive \cite{Wu2012}. 
However, here, we are interested in well-posedness for solutions on a nonzero background. In the case of the Camassa-Holm equation \eqref{CH}, this amounts to 
study the equation with the addition of dispersion since, under the change of variable $u(x,t)=v(x-kt,t)+k$, the Camassa-Holm is transformed into a version of itself with a term 
proportional to $u_x$ added on the LHS (see for example \cite[Introduction]{Danchin}). Using standard techniques, local and global well-posedness is established in \cite{Escher} for the Camassa-Holm with various terms added to the equation, including one that is proportional to $u_x$. 
In the case of the Novikov equation, under the change of variable $u(x,t)=v(x-k^2t,t)+k$, in addition to the dispersion term proportional to $u_x$, a linear combination of the three quadratic terms on the right-hand-side of the Camassa-Holm equation 
\eqref{CH} are added to the equation. While we do not have a work to cite with an existence result for this more general Novikov equation, we will assume that we have local well-posedness for initial data $v(0,\cdot)\in H^s(\RM)$ with $s>\frac{3}{2}$, and global
existence if the momentum is initially positive. Indeed, the techniques of the papers cited above \cite{Ni2011, Himonas2012, Wu2012} should apply as well, although the work would be technically ``messier'' due to the extra terms. 
 This well-posedness theory for $v\in H^s(\RM)$ with $s>\frac{3}{2}$ translates to functions $m\in X_k$ if $s=3$.   Throughout our work, we thus assume that the local and
global well-posedness holds for the space $X_k$ considered.
\end{remark}

We now show that the profile equations \eqref{e:profile1} and \eqref{e:profile_quad} satisfied by our smooth solitary waves $\mu(\cdot;k)$ correspond to 
the Euler-Lagrange equation of the action functional
\begin{equation}\label{e:Lyapunov}
\Lambda(m):=\omega_0\mathcal{E}(m)+\omega_1 F_1(m)+ F_2(m),~~m\in X_k
\end{equation}
for appropriate choices of the Lagrange multipliers $\omega_1$ and $\omega_2$.  This is accomplished in the following result.

\begin{proposition}\label{P:cp}
For a fixed $c>0$, the solitary wave $\mu(\cdot;k)$ of \eqref{e:profile1} and \eqref{e:profile_quad} is a critical point of the action functional $\Lambda$
provided we take
\begin{equation}\label{e:Lmult}
\omega_0=-9a^{-2/3}~~{\rm and}~~\omega_1=9ca^{-4/3}(2E+c),
\end{equation}
where here $a$ and $E$ are defined in terms of $k$ via \eqref{kcond}.  
\end{proposition}

\begin{proof}
Fix $c>0$ and note by straightforward calculations that the equation $\frac{\delta\Lambda}{\delta  m}(\mu)=0$ is equivalent to the integro-differential equation
\begin{equation}\label{e:der1}
2\omega_0 (1-\partial_x^2)^{-1}\mu+\frac{2\omega_1}{3\mu^{1/3}}-2\mu^{-8/3}\mu''-6\mu^{-5/3}+\frac{8\mu'^2}{3\mu^{11/3}}=0.
\end{equation}
Note that if $\mu$ is a weak solution to \eqref{e:der1} then we  have $\mu\in C^\infty(\RM)$, meaning to prove our result
it is sufficient to show that the above integro-differential equation is satisfied whenever $\mu$ satisfies the profile equation  \eqref{e:profile1}-\eqref{e:profile_quad}
and the constants $\omega_0$ and $\omega_1$ are chosen as stated.

To continue, we rewrite everything in \eqref{e:der1} in terms of the $\phi$ variable, replacing $\mu=\phi-\phi''$ and 
using the profile equations \eqref{e:profile1}-\eqref{e:profile_quad}  to express all derivatives of
$\phi$ with functions of $\phi$ only.  The calculations are messy, but after simplifications one finds that \eqref{e:der1} can be rewritten as
\[
6\omega_0a^{5/3}\phi+2\omega_1a^{4/3}(c-\phi^2)^{1/2}-18\left(c(2E+c)(c-\phi^2)^{1/2}-3a\phi\right)=0.
\]
By grouping the terms $(c-\phi^2)^{1/2}$ and $\phi$ separately, this immediately gives the unique choice for the Lagrange multipliers given in \eqref{e:Lmult}.
\end{proof}

Before we continue, it is important to note that, even though the action functional $\Lambda$ has seemingly two Lagrange multipliers, the fact that the underlying solitary waves
depend only on the single parameter $k\in(0,\sqrt{c}/2)$ implies that the Lagrange multipliers are smooth functions of only the the single variable $k$.  
Further, we note that first variations of the functionals $\mathcal{E}$, $F_1$ and $F_2$  acting on $X_k$ are not defined independently of each other due to the non-zero
boundary conditions at infinity.

Proposition \ref{e:Lmult} establishes that the smooth solitary waves $\mu(\cdot;k)$ constructed in Theorem \ref{exth} are critical points of the functional \eqref{e:Lyapunov}, where the 
Lagrange multipliers $\omega_0$ and $\omega_1$ are defined in terms of $\mu$ via \eqref{e:Lmult}.  Our next goal is to attempt to characterize the nature of the critical
point $\mu$ as a local minimum, maximum or a saddle point.  To this end, we note that since $\mu\in X_k$ is uniformly bounded below by $k>0$ by construction, we have that
\[
\mu(\cdot;k)+\widetilde{m}\in X_k
\]
provided that $\|\widetilde{m}\|_{H^1(\RM)}$ is sufficiently small.  With this in mind, we have the following result.

\begin{corollary}\label{C:expand}
There exists a constant $\eps_0>0$ sufficiently small such that for every $\widetilde{m}\in H^1(\RM)$ satisfying $\|\widetilde{m}\|_{H^1}\leq\eps_0$ we have
\eqnn{
\Lambda(\mu+\widetilde{m})-\Lambda(\mu)=\left<\mathcal{L}\widetilde{m},\widetilde{m}\right>+\mathcal{R}(\widetilde{m})
}
where here $\mathcal{L}:=\frac{\delta^2\Lambda}{\delta m^2}(\mu)$ is the closed, densely defined linear operator on $L^2(\RM)$ given explicitly by
\begin{equation}\label{Ldef}
\begin{aligned}
\mathcal{L}&=\frac{2}{3}\left(-3\frac{\partial}{\partial x}\left(\mu^{-8/3}\frac{\partial}{\partial x}\right)+8\mu''\mu^{-11/3}-\frac{44}{3}(\mu')^2\mu^{-14/3}+15 \mu^{-8/3}\right)\\
&\qquad-\frac{2\omega_1}{9}\mu^{-4/3}+2\omega_0(1-\partial_x^2)^{-1} 
\end{aligned}
\end{equation}
and where $\mathcal{R}(\widetilde{m})$ is  the remainder term satisfying $\|\mathcal{R}(\widetilde{m})\|_{H^1}\leq C_0\|\widetilde{m}\|_{H^1}^3$ for some 
$\widetilde{m}$-independent positive  constant $C_0>0$.
\end{corollary}

It follows that the nature of the critical point $\mu$ may be determined from the spectral properties of the linear operator $\mathcal{L}$ acting on $L^2(\RM)$.
Of course, since $\mathcal{L}$ contains the nonlocal term $(1-\partial_x^2)^{-1}$ the analysis of its spectrum, especially its point spectrum, is highly nontrivial\footnote{Note
one can obviously remove the nonlocal term by applying the bijective operator $(1-\partial_x^2)^{-1}$ to the operator $\mathcal{L}$.  This, however,
yields a fourth-order differential operator whose spectral theory is again highly nontrivial.}.

\section{Spectral Properties of $\mathcal{L}$}\label{S:SpecL}

The goal of this section is to study the spectrum of the linear operator $\mathcal{L}$ given in \eqref{Ldef} when considered as an operator
on $L^2(\RM)$.  By using a mix of analytical and numerical tools we provide strong evidence that the following assumption holds
for all of the smooth solitary wave solutions  in Theorem \ref{exth} of the Novikov equation.

\begin{assumption}\label{H1}
The spectrum of the symmetric operator $\mathcal{L}$, considered as a closed, densely defined linear operator on $L^2(\RM)$, consists of a simple negative eigenvalue,
a simple eigenvalue at the origin $\lambda=0$, and the rest of the spectrum is strictly positive and uniformly bounded away from $\lambda=0$.
\end{assumption}

Note that by translation invariance we clearly have
\[
\mathcal{L}\mu'=0,
\]
and hence $\lambda=0$ is indeed an eigenvalue for $\mathcal{L}$.  In Section \ref{51}, we provide some additional analytical results regarding the spectrum of $\mathcal{L}$.
More precisely, we determine the essential spectrum of $\mathcal{L}$ and obtain a lower bound for the spectrum of $\LM$ (see Proposition \ref{Bound1}).  
In Section \ref{52}, we use both of these results and numerical computations of the Evans to obtain strong evidence to that the operator $\LM$ has a one-dimensional kernel and only has one negative eigenvalue which is simple.

\subsection{Analytical Results on the Spectrum of $\LM$}
\label{51}

Throughout this section, fix $c>0$ and let  $\mu=\mu(\cdot;k)$ denote a smooth solitary wave solution to the Novikov equation, as constructed in Theorem \ref{exth}, 
with asymptotic value $k\in(0,\sqrt{c}/2)$.
Note that the linear operator $\mathcal{L}$ is closed on $L^2(\RM)$ with densely defined domain $H^2(\RM)$ and, further, it can be decomposed as
\begin{equation}\label{Ldef2}
\mathcal{L}=\mathcal{L}_0+2\omega_0(1-\partial_x^2)^{-1},
\end{equation}
where $\mathcal{L}_0$ is a Sturm-Liouville operator with smooth and bounded coefficients given by
\begin{equation}\label{LZdef}
\mathcal{L}_0=\frac{2}{3}\left(-3\frac{\partial}{\partial x}\left(\mu^{-8/3}\frac{\partial}{\partial x}\right)+8\mu''\mu^{-11/3}-\frac{44}{3}(\mu')^2\mu^{-14/3}+15 \mu^{-8/3}\right)
-\frac{2\omega_1}{9}\mu^{-4/3}.
\end{equation}
In the first result from this section, we analytically determine the essential spectrum of $\mathcal{L}$.

\begin{lemma}
\label{ess}
The essential spectrum of $\LM$ is positive, bounded away from the origin and given by the interval
\eqnn{
\sigma_{\text{ess}}(\LM)=\left[\sigma_0,\infty\right),~~{\rm where}~~\sigma_0:=\frac{8(c-4 k^{2})}{k^{\frac{8}{3}}\left(c-k^{2} \right)}>0.
}
\end{lemma}

\begin{proof}
 The essential spectrum is defined \cite[Definition 2.2.3]{KP} to be the values of $\lambda$ such that the operator
 $$
 \LM-\lambda I
 $$
 considered on $L^2(\RM)$  is either not Fredholm or Fredholm with nonzero Fredholm index. The properties of an operator of not being Fredholm 
 or being  Fredholm with nonzero Fredholm index is not altered by the application of an invertible operator. Thus, for our case, we have that
 $$
\lambda\in \sigma_{\text{ess}}(\LM)\iff (1-\partial_{x}^2)(\LM-\lambda I)\text{ is not Fredholm or is Fredholm with nonzero Fredholm index.}
 $$
 
 Now, using the decomposition \ref{Ldef2} we note that since the operator
 $$
 (1-\partial_{x}^2)(\LM-\lambda I)=(1-\partial_{x}^2)(\LM_0-\lambda)+2\omega_0
 $$
 is a linear differential operator with asymptotically constant coefficients\footnote{Note here we are also using that the coefficients approach their 
 asymptotic value as $x\to\pm\infty$ at exponential rates.}, we can use the classical result of Henry \cite[Theorem A.2]{Henry81} which states that the essential spectrum is found by computing the spectrum of the asymptotic eigenvalue problem\footnote{See also \cite[Theorem 3.1.11]{KP}.}.  In particular, we have
 \begin{equation}\label{constess}
 \lambda\in\sigma_{\text{ess}}(\LM)\iff \left((1-\partial_{x}^2)(\LM_0^\infty-\lambda)+2\omega_0\right)v=0\text{ has a non-trivial bounded solution},
\end{equation}
 where $\LM_0^\infty$ is the constant-coefficient differential operator obtained by applying the limit $|x|\to\infty$ to $\LM_0$.
 Rewriting the operator $\mathcal{L}_0$ defined in \eqref{LZdef} as
\eq{
\LM_0=\partial_x\!\left(F\partial_x\right)+G,
}{LFG}
where 
\eq{
F:=-2\mu^{-8/3}\text{ and }G:=\frac{2}{9}\left(24\mu''\mu^{-11/3}-{44}\mu'^2\mu^{-14/3}+45\mu^{-8/3}-\omega_1\mu^{-4/3}\right),
}{FGdef}
we see that the limits as $|x|\to\infty$ of the expressions above are found by simply inserting the value $\mu=k$.  Additionally
using the expressions \eqref{kcond} to express the Lagrange multipliers $\omega_0$ and $\omega_1$ in \eqref{e:Lmult} as explicit
functions of $k$, we find that
\begin{equation}\label{FGinf}
\lim_{x\to\pm\infty}F(x) = F_\infty:=-2k^{-8/3}\text{ and }\lim_{x\to\pm\infty}G(x)=G_\infty:=\frac{8c-14 k^{2}}{\left(c-k^{2} \right) k^{\frac{8}{3}}}.
\end{equation}
It follows from \eqref{constess} that the essential spectrum of $\mathcal{L}$ thus consists of all values $\lambda\in\CM$ such that the ODE
\[
(1-\partial_x^2)\left(F_\infty\partial_x^2+G_\infty-\lambda\right)v + 2\omega_0 v=0
\]
has a non-trivial bounded solution.  Being constant coefficient, one can now use Fourier analysis to find that the essential spectrum 
consists of all the values of $\lambda\in\CM$ such that
$$
(1+r^2)(G_\infty-r^2F_\infty-\lambda)+2\omega_0=0,
$$
for some $r\in\R$.  Solving for $\lambda$, this implies that the essential spectrum consists of the image of the function $\lambda:\RM\to\CM$ defined explicitly by
$$
\lambda(r)=G_\infty-r^2F_\infty+\frac{2\omega_0}{1+r^2}.
$$
Note that $\lambda(r)$ is clearly an even, real-valued function tending to positive infinity as $r\to\infty$.  Further, noting that $F_\infty$ and $\omega_0$ are both
negative  due to \eqref{FGinf} and \eqref{e:Lmult}, the global minimum value of $\lambda(r)$ necessarily occurs
at $r=0$ with
$$ 
\lambda(0)=G_\infty+2\omega_0=\frac{8(c-4 k^{2})}{k^{\frac{8}{3}}\left(c-k^{2} \right)},
$$
which we note is strictly positive for all $k\in(0,\sqrt{c}/2)$.  This completes the proof.
 \end{proof}

We now turn to studying the point spectrum of $\mathcal{L}$, which we know consists of isolated eigenvalues of finite multiplicity.  
Our first main analytical result in this direction is Proposition \ref{Bound1}, which gives a lower bound on the point spectrum of $\LM$ that will end up being useful in Section \ref{52} 
for our numerical computations. To obtain this result, we write the operator $\LM$ as a sum of two operators, see \eqref{OSum} below,
and use the fact that  we are able to find a lower bound for the spectra of each of the two terms in that sum. 
For completeness, we also provide a bound obtained from energy estimate computations (see Proposition \ref{EEs}). However, in the next section, that bound is found to be numerically much larger than the one provided by by Proposition \ref{Bound1}.

\begin{lemma}
\label{Specsame}
The kernels of the operator $\mathcal{L}$ given in \eqref{Ldef} and of the operator
\eq{
\widetilde{\mathcal{L}}=\mathcal{L}_0+{2\omega_0}f,\;\;f:= \frac{a^{2/3}}{3\phi \mu^{5/3}},
}{LTdef}
have a nontrivial intersection, that is 
$$
\mu'\in\ker\!{(\widetilde{\mathcal{L}})}\cap \ker\!{({\mathcal{L}})},
$$
where $\mu$ is the traveling wave solution given in Theorem \ref{exth}.
\end{lemma}
\begin{proof}
The relation 
\eq{\mathcal{L}\mu'=0}{rel2}
is a consequence of the invariance of the Novikov equation with respect to space translation. It can be checked in two ways. The simplest is to 
use the fact that $\LM$ is the second Fr\'echet derivative of the functional $\Lambda$ evaluated at the traveling wave solution. As a consequence, the relation  \eqref{rel2} is obtained by differentiating \eqref{e:der1} with respect to $x$. The more direct way is to apply the operator 
$\LM$ to $\mu'$ and use the profile equations \eqref{e:profile1}-\eqref{e:profile_quad} to show that we obtain \eqref{rel2}.

To show that $\mu'\in\ker\!{(\widetilde{\mathcal{L}})}$, we rewrite the relation \eqref{e:profile1} for the traveling wave solution $m=\mu(x-ct)$ and $u=\phi(x-ct)=(1-\partial_x^2)^{-1}\mu$ and the constant $a=D^{3/2}$ as
$$
\phi^2-c=\frac{a^{2/3}}{\mu^{2/3}}.
$$
Differentiating the equation above with respect to $x$, we obtain
$$
2\phi (1-\partial_x^2)^{-1}\mu'=-\frac{2a^{2/3}}{3\mu^{5/3}}\mu'.
$$
The operators $\LM$ and $\widetilde{\LM}$ defined in \eqref{Ldef} and \eqref{LTdef} differs only in their last term.  The relation above shows that the two last terms, when applied to $\mu'$, are the same. We thus have that $\mu'$ also is in $\ker(\widetilde{\LM})$. 
\end{proof}

\begin{remark}
In essence, the point of the above proof is that while $\mathcal{L}$ is non-local, its application to $\mu'$ gives precisely $\widetilde{\mathcal{L}}\mu'$, i.e. the
operators $\mathcal{L}$ and $\widetilde{\mathcal{L}}$ agree when acting on $\mu'$.  Since $\widetilde{\mathcal{L}}$ is a local, Sturm-Liouville operator, its point
spectrum is arguably easier to study than that of $\mathcal{L}$.
\end{remark}

For the next lemma, we consider the operator arising from the difference $\LM-\widetilde{\LM}$. To do so, we obtain the following 
expression for the function $f$ defined in \eqref{LTdef}
\eq{
f= \frac{a^{2/3}}{3\phi \mu^{5/3}}= \frac{(c-\phi^2)^{5/2}}{3a \phi},
}{fdef}
where the last equality was obtained by applying the profile equation \eqref{e:profile1} with $m=\mu$ and $u=\phi$.

\begin{lemma}
\label{Lemspec}
The spectrum of the operator
\begin{equation}\label{Sdef}
\mathcal{S}:=f-(1-\partial_x^2)^{-1}
\end{equation}
acting on $L^2(\RM)$ is real and satisfies
\begin{equation}\label{Sspec}
\sigma\left(\mathcal{S}\right)\subset \left[f_0-1,f_\infty\right],
\end{equation}
where here
\[
f_0:=f(0)\text{ and }f_\infty:=\lim_{x\to\infty}f(x),
\]
with the function $f$ defined in \eqref{fdef} and
\eq{0<f_0<1\text{ and }f_\infty>1.}{fcond} 
\end{lemma}
\begin{proof}
It can be checked from the last expression in \eqref{fdef} that $f$ is a decreasing function of $\phi$. Thus the supremum
value of $f$ is $f_\infty$, while its minimum value is attained at $x=0$. As a consequence, the spectrum of the multiplication operation by $f$ consists 
of the range of $f$ given by the interval $\left[f_\infty,f_0\right]$. 

Further, the operator $(1-\partial_x^2)^{-1}$ can be written as a convolution as
\eq{
(1-\partial_x^2)^{-1}v=p(x)\ast v,\;\;p(x):=\frac{e^{-|x|}}{2},
}{convw}
since $p(x)$ as defined above is the Green's function for the operator $(1-\partial_x^2)$. Taking the Fourier transform of this operator 
transforms convolution into multiplication by $\hat{p}$. Now, for the multiplication operator, the spectrum is the range of of the Fourier transform $\hat{p}$.
Since the Fourier transform is unitary, we get that spectrum of $(1-\partial_x^2)^{-1}$
 is also the range of $\hat{p}$, which can be checked to be $[0,1]$. Given the form of the self-adjoint operator $\mathcal{S}$ written as a difference between two 
 self-adjoint operators whose spectrum has been determined above, the spectrum of $\mathcal{S}$ lies in the difference of the two intervals as given in \eqref{Sspec}.
 

Finally, it remains to establish the bounds \eqref{fcond}.  To this end, note that since 
$\phi\to k$ as $|x|\to\infty$ we have
 \eqnn{
 f_\infty=\frac{(c-k^2)^{5/2}}{3ak}=\frac{c-k^2}{3k^2}>\frac{c-c/4}{3c/4}=1,
 }
 where we used the definition of $f_\infty$ for the first equality, \eqref{e:profile_quad} to express $a=a(k)$ for the second equality, 
 \eqref{kcond2} for the inequality, and \eqref{kcond} for the final equality. For the bound on $f_0$, we use 
 \eqref{fdef} evaluated at $x=0$ to find
 \eqnn{
 f_0&=\frac{(c-\phi(0)^2)^{5/2}}{3a\phi(0)}=\frac{c-\phi(0)^2}{3\left(E+\phi(0)^2/2\right)}<\frac{3c/4}{3\left(E+c/4\right)}<1,
 }
 where here we used the definition of $f$ for the first equality, \eqref{e:profile_quad} evaluated at $x=0$ for the second equality, 
 the fact that $\phi(0)>\sqrt{c}/2$ for the first inequality, and the positivity of $E$ for the last inequality.
\end{proof}

Next, we use the above to establish a lower bound on the spectrum of $\LM$, which we use for our numerical computations.

\begin{proposition}
\label{Bound1}
The spectrum of $\LM$ is bounded from below by a negative number and is included in the following interval 
\begin{equation}\label{Boundexp} 
\sigma(\LM)\subset \left[\sigma_1,\infty\right),~~{\rm with}~~\sigma_1:=\lambda_-(\widetilde{\LM})+2\omega_0(1-f_0)< 0,
\end{equation}
where here $f_0$ is the value of the expression given in \eqref{fdef} at $x=0$, $\lambda_-(\widetilde{\LM})$ is the unique 
negative eigenvalue of the operator $\widetilde{\LM}$ given in \eqref{LTdef}, and $\omega_0$ is given in \eqref{e:Lmult}.
\end{proposition}
\begin{proof}
The operator $\widetilde{\LM}$ given in \eqref{LTdef} is a Sturm-Liouville operator with bounded coefficients defined on the whole line. As stated in Lemma \ref{Specsame}, the zero eigenvalue corresponds to the eigenvector $\mu'$, which has one zero at the origin. This implies that one and only one eigenvalue is negative, which we denote by $\lambda_-(\widetilde{\LM})$. We have that the self-adjoint operator $\LM$ can be written as the sum of two other self-adjoint operators as 
\eq{
\LM=\widetilde{\LM}-2\omega_0\mathcal{S},}{OSum} 
with $\mathcal{S}$ defined in \eqref{Sdef}. Since, from Lemma \ref{Specsame}, the spectrum of $\mathcal{S}$ is bounded from below by $f_0-1$, the spectrum of $\widetilde{\LM}$  is bounded from below $\lambda_-(\widetilde{\LM})$, and since $\omega_0<0$, we have that the spectrum of $\LM$ is bounded from below by $\lambda_-(\widetilde{\LM})-2\omega_0(f_0-1)$, as specified by the proposition.
\end{proof}

\begin{remark}
We note that while one may hope to use the decomposition of $\mathcal{L}$ provided in \ref{OSum} to analytically verify Assumption \ref{H1}, this
turns out not to work.  See Remark \ref{R:pos} in the next section.  
\end{remark}

Proposition \ref{Bound1} gives a lower bound gives a bound on the spectrum of $\LM$, which will be useful for our numerical computations in Section \ref{52} below.
We note that, as an alternative to the approach to the lower bound taken above, one could also attempt to achieve such a bound through
the use of energy estimates.  For completeness, next we provide the result of applying direct energy estimates.  As we will see in our
numerical calculations, however, the bound obtained in Proposition \ref{Bound1} is in fact much better: see Remark \ref{bad_bound} below.

%
%

\begin{proposition}
\label{EEs}
Let $\lambda$ be an eigenvalue for the operator $\LM$ given in \eqref{Ldef}. Then it satisfies the inequality
\eq{
\lambda\geq 2\omega_0 -\sup_{x\in\RM}|G(x)|,
}{inf}
where here the function $G$ is given in \eqref{FGdef} and $\omega_0$ is the (necessarily negative) Lagrange multiplier given in \eqref{e:Lmult}.
\end{proposition}

\begin{proof}
Let $\lambda$ be an eigenvalue for $\mathcal{L}$ with eigenfunction $v\in L^2(\RM)$ and note that writing $\mathcal{L}$ as in \eqref{Ldef2}
with $\mathcal{L}$ expressed as in \eqref{LFG}, it follows that
 \eq{
 \partial_x(F\partial_x)v+Gv+2\omega_0(1-\partial_x^2)^{-1}v=\lambda v.
 }{eigin}
Multiplying \eqref{eigin} by $\overline{v}$, integrating over $\R$, using \eqref{convw} to rewrite  the last term on the RHS as a convolution, and performing integration by parts we obtain
$$
\lambda \|v\|_{L^2}^2=-\left<Fv_x,\,v_x\right>+\left<Gv,\,v\right>+2\omega_0\left<p\ast v,\,v\right>,
$$
where here $\|\,.\,\|$ and $\left<\cdot,\cdot\right>$ denote the $\LT$ norm and inner-product, respectively.
Since $F<0$ and $\omega_0<0$, we obtain the following inequality
\begin{align*}
\lambda \|v\|_{L^2}^2&\geq -\left|\left<Gv,\,v\right>\right|+2\omega_0\left|\left<p\ast v,\,v\right>\right|\\
&\geq -\sup_{x\in \R}\left(|G|\right)\|v\|_{L^2}^2+2\omega_0\|p\ast v\|_{L^2}\|v\|_{L^2}.
\end{align*} 	
Noting that Young's convolution inequality implies that
\begin{equation}\label{e:Young}
\|p\ast v\|_{L^2}\leq \|p\|_{L^1} \| v\|_{L^2}=\|v\|_{L^2},
\end{equation}
where here we used the definition of $p$ given in \eqref{convw} and computed that $\|p\|_{L^1(\mathbb{R})}=1$, we arrive at the stated inequality
\eqref{inf}.
\end{proof}

\subsection{Evans Function Computations}
\label{52}

In this section, we augment the analytical results from the previous section with a numerical investigation of the point spectrum of $\mathcal{L}$.
In particular, we aim to provide well-conditioned numerical evidence that Assumption \ref{H1} holds for all smooth solitary waves of the 
Novikov equation \eqref{e:nov2}.

To this end, note that while the operator $\LM$ as defined in \eqref{Ldef} is nonlocal, the eigenvalue problem $\LM v=\lambda v$ can be 
written as an equivalent ordinary differential equation by applying the invertible operator $1-\partial_x^2$.  Specifically, we have
\begin{equation}\label{evdelta2} 
\lambda\text{ is an eigenvalue of $\LM$}\iff (1-\partial_x^2)\left(\mathcal{L}_0-\lambda\right)v+2\omega_0 v=0\text{ has a solution $v\in\LT\setminus\{0\}$},
\end{equation}  
where $\mathcal{L}_0$ is given in \eqref{LZdef}. From here on, we refer to \eqref{evdelta2} as our {\sl{eigenvalue problem}}. 
Our goal is to perform numerical computations of the Evans function associated with the eigenvalue problem in \eqref{evdelta2}. 
For details about the Evans function $D(\lambda)$, see \cite{KP}, but for our purposes the key things to know is that 
$D(\lambda)$ is a complex analytic function away from the essential spectrum of $\mathcal{L}$ whose roots agree in both location and multiplicity
with the eigenvalues for $(1-\partial_x^2)\mathcal{L}$.  Our general strategy is to first use the lower bound 
provided by Proposition \ref{Bound1} to determine a region of the real axis that necessarily contains
any negative eigenvalue of $\mathcal{L}$.  We then enclose this region in the complex plane by a simple closed contour $\Gamma$ and perform a numerical
winding number computation of the Evans function around this contour, being sure to use Lemma \ref{ess} to avoid the contour intersecting the essential spectrum of $\mathcal{L}$.

In the following, we briefly review the construction of the Evans function as well as our numerical methods.

\subsubsection{Computation of Smooth Solitary Wave Solutions}

\begin{figure}[t!]
\begin{center}
	\includegraphics[width=0.7\textwidth]{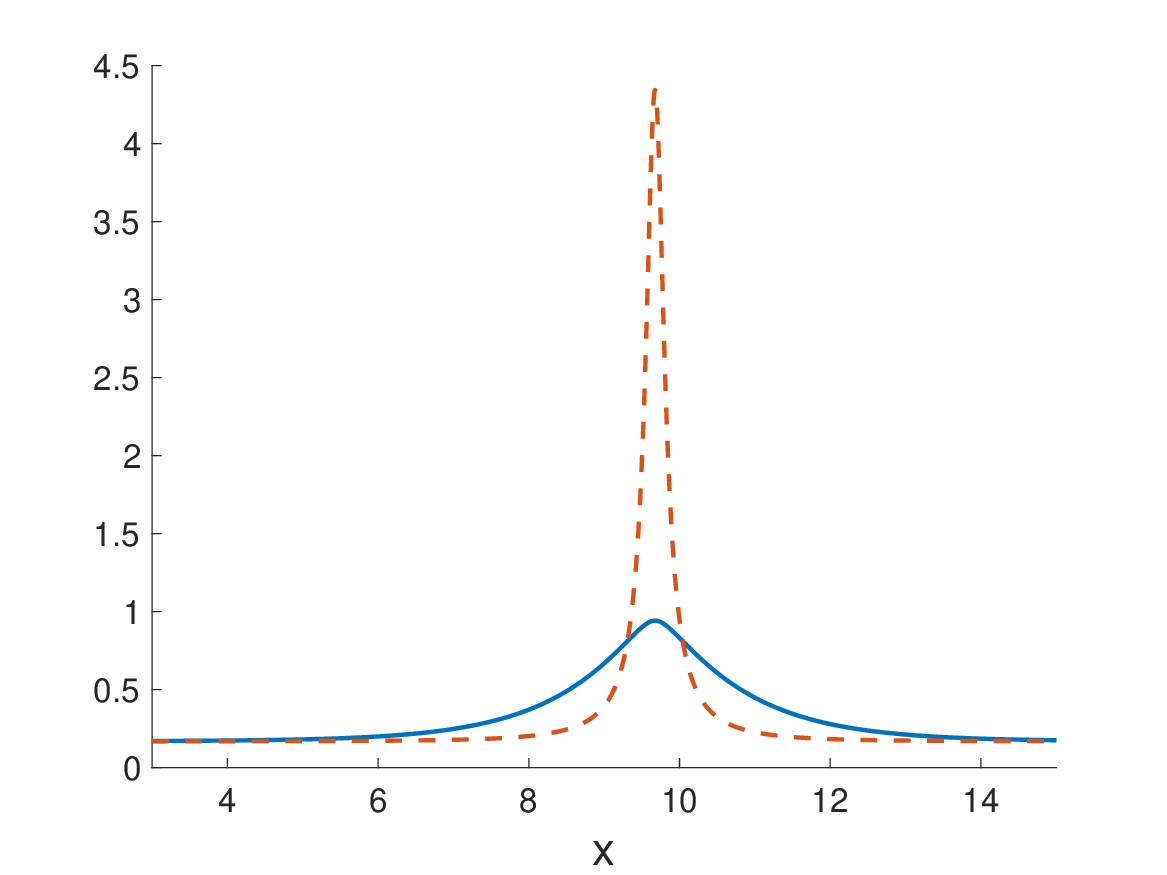}
	\caption{The solid curve shows the plot of $\phi$ as a function of $x$ obtained though our shooting method for $c=1$ and $a=3\sqrt{3}/32$. The dashed line shows the graph of $\mu$ obtained from $\phi$ through the profile equation \eqref{e:profile1} as $\mu=a/(c-\phi)^{3/2}$. \label{Shootingfigure}}
	\end{center}
\end{figure}

Throughout, we fix $c>0$.  To numerically compute smooth solitary wave profile $\phi$, we start by writing the profile equation \eqref{e:profile_quad}
satisfied by $\phi$ as the two-dimensional dynamical system
\[
\left\{\begin{aligned}
u_{1}'&=u_2\\
u_{2}'&=u_1- \frac{a}{(c-u_1^2)^{3/2}},
\end{aligned}\right.
\]
where here we recall that the integration constant $a$ satisfies $0<a<\frac{3\sqrt{3}c^2}{16}$.  For each such $a$, we aim to compute
the homoclinic orbit that connects the fixed point $(0,k)$ to itself, where, for given values of $a$ and $c$, the asymptotic end-state $k$ is determined by the smallest
of the two solutions of \eqref{kcond}(i) on the interval $(0,c)$.  One can easily check that the (linear) unstable direction at the point $(0,k)$ is given by
$(1,C(k))$, where here $C(k)$ is the corresponding positive eigenvalue of the linearization given by $C(k)=\sqrt{(c-4k^2)/(c-k^2)}$.
Using the MatLab program {\sl{ODE45}},  homoclinic orbits for a given value of $0<a<\frac{3\sqrt{3}c^2}{16}$ are now obtained by using 
simple shooting method from the point $(0,k)$.  The result of our computation for the case $c=1$ and $a=\frac{3\sqrt{3}c^2}{32}$ is given in Figure \ref{Shootingfigure}.

%

%

\subsubsection{Construction of the Evans Function}

%
%

We now briefly recall the analytical construction of the Evans function.  To begin, we first rewrite the eigenvalue problem \eqref{evdelta2} as a
four-dimensional, first-order linear system of the form
%
%
\begin{equation}\label{linear}
\frac{d \vec{U}}{dx}  = A(x, \lambda)\vec{U}, 
\end{equation}
where $A$ is the $4\times 4$ matrix given by 
\eqnn{ A(x, \lambda) = 
\begin{pmatrix}
 0&1  &0&0\\
  0&0  &1&0\\
  0&0  &0&1\\
\left(G-G''+2\omega_0-\lambda\right)/F&\left(F'-2G'-F''\right)/F&\left(F-G-3F''+\lambda\right)/F&-3F'/F
\end{pmatrix},
}
where the functions $F$ and $G$ are given explicitly in \eqref{LFG}.  As are seeking solutions to our eigenvalue
problem \eqref{evdelta2} that are in $L^2(\RM)$, we must ensure that our solutions to \eqref{linear} tend to zero
as $x\to\pm\infty$.  With this in mind, we note that the asymptotic behavior as $x\rightarrow\pm\infty$ of  solutions to (\ref{linear}) 
is determined by the asymptotic matrix
\begin{equation}\label{Am2}
A^{\infty}(\lambda):=\lim_{x\rightarrow \pm\infty}A(x,\lambda)
=
\left(\begin{array}{cccc}
 0&1  &0&0\\
  0&0  &1&0\\
  0&0  &0&1\\
\frac{1}{2}\left(\lambda k^{8/3}-4+\frac{24}{c-k^2}\right)&0&\frac{1}{2}\left(10-\lambda k^{8/3}-\frac{6k^2}{c-k^2}\right)&0
\end{array}\right)
\end{equation}
which is found by using the explicit expressions for the functions $F$ and $G$
in \eqref{FGdef} and  inserting the values $\mu=k$ and $\partial_x^\ell\mu=0$ for $\ell\geq 1$ in $A$.
For $\lambda$ not in the essential spectrum $\sigma_{\rm ess}(\mathcal{L})$ of the operator $\mathcal{L}$, as determined in Lemma \ref{ess},
the asymptotic matrix $A^\infty(\lambda)$ will have two eigenvalues with positive real part, and two with negative real part.  Indeed, 
the essential spectrum is precisely the values $\lambda\in\CM$ for which $A^\infty(\lambda)$ has an eigenvalue on the imaginary axis.  
It follows that the number of eigenvalues of $A^\infty(\lambda)$ with positive and negative  real parts must be constant on 
the connected set $\mathbb{C}\setminus\sigma_{\rm ess}(\mathcal{L})$.  Further, noting that Lemma \ref{ess} implies that $0\notin\sigma_{\rm ess}(\mathcal{L})$,
we can directly calculate that the eigenvalues of the matrix $A^\infty(0)$ are given explicitly by
\begin{equation}\label{eigplus}
\mu_{1}^\pm(0)=\pm 2,\quad  \mu_{2}^\pm(0)=\pm\sqrt{\frac{c-4k^2}{c-k^2}}
\end{equation}
and hence, by the above discussion, it follows that $A^\infty(\lambda)$ must have two eigenvalues $\mu_i^+(\lambda)$ with positive real real parts and two 
eigenvalues $\mu_i^-(\lambda)$ with negative real parts for all $\lambda\notin\sigma_{\rm ess}(\mathcal{L})$, as claimed.
Finally, for $\lambda\notin\sigma_{\rm ess}(\mathcal{L})$ we let $v_i^+$ denote the eigenvectors with unit norm of $A^\infty(\lambda)$ corresponding to
the eigenvalues $\mu_i^+$, and similarly $v_i^-$ will denote eigenvectors with unit norm of $A^\infty(\lambda)$ corresponding to the eigenvalues $\mu_i^-$.

Now, by asymptotic ODE theory it follows for $\lambda\notin\sigma_{\rm ess}(\mathcal{L})$ that the nonlinear system \eqref{linear} will
have two linearly independent solutions $U_1^+$ and $U_2^+$ converging to zero as $x\to\infty$ and two linearly independent
solutions $U_1^-$ and $U_2^-$ converging to zero as $x\to-\infty$, which satisfy
\[
\lim_{x\to-\infty}U_i^+e^{-\mu_i^+ x}=v_i^+~~{\rm and}~~\lim_{x\to\infty}U_i^-e^{-\mu_i^- x}=v_i^-.
\]
It follows that a given $\lambda_0\notin\sigma_{\rm ess}(\mathcal{L})$ is an eigenvalue for \eqref{evdelta2} if and only if the 
space of solutions tending to zero as $x\to-\infty$, spanned by $U_1^+$ and $U_2^+$, and the space of solutions tending two zero
as $x\to\infty$, spanned by $U_1^-$ and $U_2^-$, have a nontrivial intersection of strictly positive dimension when $\lambda=\lambda_0$.
It follows that the eigenvalues of $\mathcal{L}$ in $\CM\setminus\sigma_{\rm ess}(\mathcal{L})$ are precisely the roots of the function\footnote{Note the vanishing 
or non-vanishing of the given determinant at a given $\lambda\notin\sigma_{\rm ess}(\mathcal{L})$
is independent of $x$, and hence the right-hand-side can be calculated at any convenient $x\in\RM$ chosen (typically, one takes $x=0$).}
\eqnn{
D(\lambda)=\det\left(U_1^+(\lambda),U_2^+(\lambda),U_1^-(\lambda),U_2^-(\lambda)\right)\big{|}_{x=0},
}
which is known as the Evans function\cite{AGJ,Evans, Gardner98,Jones, Kapitula98a,Kapitula00, Li00, Pego, Sandstede,Yanagida}.  

\begin{remark}
We note that the two solutions  $U_{1}^-$ and $U_{2}^-$ may be numerically obtained by integrating (\ref{linear}) backwards from a sufficiently large positive value of $x$, with initial 
conditions in the $v_{1}^-$ and $v_{2}^-$ directions, respectively.  We note, however, that even though the two eigendirections $v_1^-$ and $v_2^-$ 
are linearly independent, the numerical integration
will lead to an alignment with the eigendirection corresponding to the
eigenvalue with smallest real part. One way to circumvent this problem is to compute  the Evans function using the alternative 
definition involving exterior algebra
\cite{AfBr01, AlBr02, Br99, BrDeGo02, Br00, DeGo05, NgR79, skms, EvansWeberGroup}.  This is the approach we take here.
\end{remark}

In addition to the above, it can be shown that  the zeroes of $D$ on $\CM\setminus\sigma_{\rm ess}(\mathcal{L})$ agree in location and algebraic
multiplicity with to the eigenvalues of \eqref{evdelta2}.  Further, the Evans function $D(\lambda)$ above is complex-analytic on $\CM\setminus\sigma_{\rm ess}(\mathcal{L})$
and is real-valued if $\lambda$ is real.  See, for example, the references cited above of \cite{KP}.  In what follows, we discuss results
coming from the numerical evaluation of the Evans function defined above.

\subsubsection{Numerical Evans Function Calculations}

In this section, we discuss our numerical investigation of the Evans function $D(\lambda)$.
Throughout our work, we perform the numerical Evans function computations using the MATLAB-based numerical library for Evans function computation called STABLAB \cite{StabLab}. 
 In STABLAB, the Evans function is computed using the {\sl{polar-coordinate method}}, a method initially proposed by Humpherys and Zumbrun in \cite{Humpherys06}, which represents the unstable and stable manifolds using the continuous orthogonalization method of Drury \cite{Drury80} together with a scalar ODE that restores analyticity.   Further,
we note the scaling
\[
u\rightarrow \sqrt{c}u,\;\;t\rightarrow ct,
\]
applied to the Novikov equation \eqref{e:nov2} written in traveling wave variables sets the wave speed parameter $c>0$ to $c=1$.   Thus, throughout
our work we will take $c=1$.

Now, note that to verify Assumption \ref{H1} for a given smooth solitary wave $\mu$ of the Novikov equation\footnote{In our presentation, we first discuss the methodology
for given, arbitrary wave $\mu$.  We will then discuss the application of the methodology to a selection of smooth solitary wave solutions $\mu$.}, 
we must determine the number of negative eigenvalues
of the associated $\mathcal{L}$ and also determine the dimension of the kernel of $\mathcal{L}$.  To this end, For a given wave $\mu$ we first aim to use Proposition \ref{Bound1}
to determine a lower bound for the point spectrum of $\mathcal{L}$.  Of course, since the bound in \eqref{Boundexp} includes the unique negative eigenvalue
$\lambda_-(\widetilde{\mathcal{L}})$ for the Sturm-Liouville operator $\widetilde{\mathcal{L}}$, we first need to numerically study the negative spectrum of $\widetilde{\mathcal{L}}$.
To this end, we note that an Evans function for $\widetilde{\mathcal{L}}$ can be defined as in the previous section, albeit this time the determinants are only two-by-two 
rather than four-by-four.  Call this Evans function $\widetilde{D}$ and note, since $\lambda=0$ is an eigenvalue for $\widetilde{\mathcal{L}}$ that we must 
have $\widetilde{D}(0)=0$.  To get a lower-bound for $\lambda_-(\widetilde{\mathcal{L}})$, we define a rectangular contour $B\in\CM$
that encloses $\lambda=0$ and part of the negative real axis and compute the winding number
\[
\frac{1}{2\pi i}\oint_B\frac{\partial_\lambda\widetilde{D}(\lambda)}{\widetilde{D}(\lambda)}d\lambda
\]
along the contour $B$: see Figure \ref{F:contours}(a)  By fixing the right side of $B$ and incrementally moving the left side of $B$ farther along the negative real axis, we eventually
find that the above winding number jumps from $1$, indicating the only eigenvalue for $\widetilde{\mathcal{L}}$ inside $B$ is at $\lambda=0$,
to $2$, indicating that now $B$ encloses the unique negative eigenvalue $\lambda_-(\widetilde{\mathcal{L}})$.  This gives us a lower-bound 
for the value $\lambda_-(\widetilde{\mathcal{L}})$, which we can then use in the expression \eqref{Boundexp} to give us a lower-bound $\sigma_1$
for the lower-bound on the point spectrum of $\mathcal{L}$.

\begin{figure}[t!]
\begin{center}
(a)\includegraphics[scale=0.6]{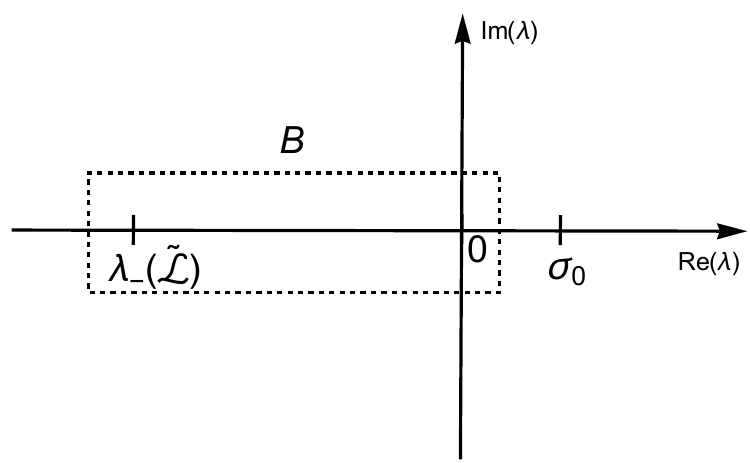}~~~(b)\includegraphics[scale=0.6]{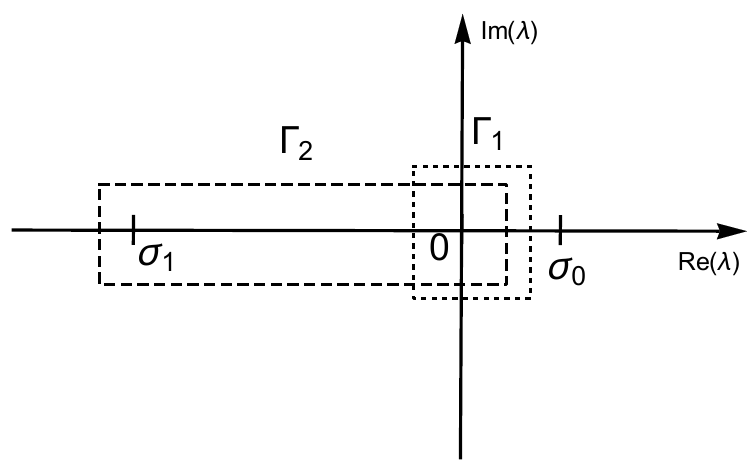}
\caption{Contours for our various numerical computations.  (a) A depiction of the contour $B$ used to numerically find a lower bound for $\lambda_-(\widetilde{\mathcal{L}})$.
(b) A depiction of the contours $\Gamma_1$ and $\Gamma_2$ used to verify the simplicity of the eigenvalue at $\lambda=0$ for $\mathcal{L}$
as well as that $\mathcal{L}$ has only one negative eigenvalue.}\label{F:contours}
\end{center}
\end{figure}

\begin{remark}\label{bad_bound}
Let us add that we have also considered the bound provided by Proposition \ref{EEs}. For each wave considered in our study, we found the value of the bound provided by Proposition   \ref{EEs} to be worse (and in most cases, more than one order of magnitude more negative) than the bound provided by Proposition \ref{Bound1}.
For example, for the wave associated to $c=1$ and $a=3\sqrt{3}/32$ we find that the lower bound given by Proposition \ref{Bound1} is $\sigma_1=-68.266$, 
while that numerically computed  from the formula in Proposition \ref{EEs} is equal to $-947.25$.  As such, throughout our work we used the bound provided by 
Proposition \ref{Bound1} in our numerical winding number calculations.
\end{remark}

Continuing, with the lower-bound $\sigma_1$ determined for a given wave $\mu$ we now turn our attention to the full operator $\mathcal{L}$ associated to $\mu$.  
To this end, we first  draw a small 
rectangular contour $\Gamma_1$ around the origin in the complex plane (see Figure \ref{F:contours}(b)), being sure to avoid intersecting with the essential spectrum,
and numerically confirming that
\[
\frac{1}{2\pi i}\oint_{\Gamma_1}\frac{\partial_\lambda D(\lambda)}{D(\lambda)}d\lambda=1,
\]
indicating that $\lambda=0$ is indeed a simple eigenvalue of $\mathcal{L}$.  We then select a larger contour $\Gamma_2$ that encloses the interval $(\sigma_1,0)$
containing all possible negative eigenvalues of $\mathcal{L}$.  Additionally, to ensure $\Gamma_2$ does not enclose any possible
positive eigenvalue of $\mathcal{L}$, we ensure that the real-positive subset of the interior of $\Gamma_2$ is also enclosed by $\Gamma_1$ above:
see Figure \ref{F:contours}(b) for a depiction.
We may then aim to numerically confirm that
\[
\frac{1}{2\pi i}\oint_{\Gamma_2}\frac{\partial_\lambda D(\lambda)}{D(\lambda)}d\lambda=2,
\]
indicating that $\mathcal{L}$ has precisely two eigenvalues (counting multiplicity) inside $\Gamma_2$.  Since $\lambda=0$ is enclosed by
$\Gamma_2$, it would then follow that $\mathcal{L}$ has precisely one negative eigenvalue, as desired.  Assuming the above holds for a given smooth
solitary wave $\mu$ this will provide strong numerical evidence that Assumption \ref{H1} holds for the wave $\mu$.

It remains to implement the above numerical procedure.  To this end, recall that we have scaled to $c=1$ and hence Theorem \ref{exth} implies that
smooth solitary waves will exist for parameter values $a\in(0,a_L)$ with $a_L:=\frac{3\sqrt{3}}{16}$.  In our calculations, we ran the above numerical computation
on the thirteen values\footnote{Note that for $j=1,2$ the wave is very nearly peaked, introducing considerable additional difficulty in our numerical computations.  It
would be interesting to numericaly investigate the small $a$ regime further.  At the very least, we note that as the limiting peakon is nonlinearly stable in $H^1$ it 
seems reasonable to expect that Assumption \ref{H1} continues to hold for nearly peaked waves.}
\[
a_j = j\left(\frac{a_L}{16}\right),~~j=3,4,5,,\ldots,15.
\]
Our numerical computations show that Assumption \ref{H1} hold for each of the associated thirteen smooth solitary wave solutions generated.  This provides strong
numerical evidence for the following.

\

\noindent
{\bf Numerical Observation:} Assumption \ref{H1} holds for all smooth solitary waves of the Novikov equation constructed in Theorem \ref{exth}.

\

We note that it is an interesting open problem to analytically verify Assumption \ref{H1} for the Novikov equation.

\begin{remark}\label{R:pos}
Note that from the decomposition of the operator $\mathcal{L}$ given in \eqref{OSum}, one may hope to analytically establish Assumption \ref{H1} noting
that the Sturm-Liouville operator $\widetilde{\mathcal{L}}$ has only one negative eigenvalue and by showing that the operator $\mathcal{S}$ is non-negative.
According to our numerical investigation\footnote{Observe that \eqref{Sspec} provides a bounded region containing $\sigma(\mathcal{S})$ and hence one can use
numerical Evan's function techniques as above to search for possible negative spectrum of $(1-\partial_x^2)\mathcal{S}$.}, however, this seems not to be the 
case as the operator $\mathcal{S}$ seems to have negative spectrum.
\end{remark}

\section{Orbital Stability Analysis}\label{S:Orbital}

Throughout this section, fix $c>0$ and let $\mu(\cdot;k)$ be a smooth solitary wave for the the Novikov equation as constructed in Theorem \ref{exth}.  Further, we will
assume throughout that the spectral Assumption \ref{H1} holds for the wave $\mu(\cdot;k)$ considered.  Under this assumption, the first goal of this section
is to derive an analytic criteria that guarantees the nonlinear orbital stability of the solitary wave $\mu$ to perturbations in $H^1(\RM)$.  Once this condition
is established, we will then provide a numerical verification of this stability criteria, demonstrating that it appears to hold for all smooth solitary
waves of the Novikov equation.    

\subsection{Coercivity of $\Lambda$}

To begin note that, note that if  Assumption \ref{H1} holds then Corollary \ref{C:expand} implies that the solitary wave $\mu$ is a 
degenerate saddle point of the action functional $\Lambda$.    Further, treating $\omega_0$ and $\omega_1$ 
as Lagrange multipliers in the definition of $\Lambda$ in \eqref{e:Lyapunov}, our solitary waves can be considered as \emph{constrained} critical points of $F_2$ subject to fixed
$\mathcal{E}$ and $F_1$.  Thus, it may still be possible establish the stability of $\mu$ provided we can show it is a constrained local minimizer of $F_2$ or, equivalently,
of $\Lambda$.  To understand precisely the appropriate constraint, we note that, formally,
\[
\frac{\delta \mathcal E}{\delta m}(\mu)=2(1-\partial_x^2)^{-1}\mu~~{\rm and}~~\frac{\delta F_1}{\delta m}(\mu)=\frac{2}{3}\mu^{-1/3},
\]
although neither of these variations individually represents an integrable function.  However, we see that the functional
\begin{equation}\label{e:F}
\mathcal{F}(m):=\mathcal{E}(m)-3k^{4/3}F_1(m)
\end{equation}
is well-defined on $X_k$ and that its first variation
\begin{equation}\label{e:F'}
\frac{\delta\mathcal{F}}{\delta m}(\mu)=2(1-\partial_x^2)^{-1}\mu-2k^{4/3}\mu^{-1/3}
\end{equation}
is an integrable function.  As $\mathcal{F}$ is built out of conserved quantities, it follows that the evolution of \eqref{e:nov2} does not occur on 
all of $X_k$ but rather on the co-dimension one submanifold
\[
\mathcal{M}_k:=\left\{m\in X_k:\mathcal{F}(m)=\mathcal{F}(\mu(\cdot;k))\right\}.
\]
Note, in particular, that the entire group orbit
\[
\mathcal{O}_k:=\left\{\mu(\cdot-x_0;k):x_0\in\RM\right\}
\]
is contained in the manifold $\mathcal{M}_k$.  The goal of this section is to demonstrate coercivity of the functional $\Lambda$ in \eqref{e:Lyapunov}
on the nonlinear manifold $\mathcal{M}_k$ in a sufficiently small neighborhood of the group orbit $\mathcal{O}_k$.  
%
As we will see, 
by Taylor's theorem it is sufficient to establish that the bilinear form associated to the symmetric linear operator $\mathcal{L}$ 
satisfies an $H^1$ coercivity bound on the tangent space to $\mathcal{M}_k$ at $\mu$.  To this end, we define
\[
\mathcal{T}_0:=\left\{v\in H^1(\RM): \left<\frac{\delta\mathcal{F}}{\delta m}(\mu),v\right>=0\right\}
\]
and note that $\mathcal{T}_0$ is precisely the tangent space in $L^2(\RM)$ to the submanifold $\mathcal{M}_k$ at the point $\mu(\cdot;k)$.
Our next result establishes a sufficient condition for the positivity of the linear operator $\mathcal{L}$ on the elements of
the tangent space $\mathcal{T}_0$ which are orthogonal to the kernel of $\mathcal{L}$.

%
%
%

\begin{lemma}\label{P:VK}
Let $\mu(\cdot;k)$ be a smooth solitary wave solution of \eqref{e:nov2} and assume that Assumption \ref{H1} holds.  If
\begin{equation}\label{e:VK}
\left<\mathcal{L}^{-1}\frac{\delta\mathcal{F}}{\delta m}(\mu),\frac{\delta\mathcal{F}}{\delta m}(\mu)\right><0
\end{equation}
then there exists a constant $\gamma>0$ such that
\begin{equation}\label{e:L_coercive}
\left<\mathcal{L}v,v\right>\geq \gamma\|v\|_{H^1}^2
\end{equation}
for all $v\in\mathcal{T}_0$ with $v\perp \mu'$.
\end{lemma}

\begin{remark}
Note that since $\mathcal{L}$ preserves parity, specifically, it maps even functions to even functions,  the Fredholm alternative implies that
the even function $\frac{\delta\mathcal{F}}{\delta m}(\mu)$ is in the range of  the symmetric operator $\mathcal{L}$ since ${\rm ker}(\mathcal{L})={\rm span}\{\mu'\}$ by Assumption \ref{H1}.
\end{remark}

\begin{proof}
It is by now well known that the operator $\mathcal{L}$ is non-negative on $\mathcal{T}_0$ provided that the condition \eqref{e:VK} holds: see,
for example, \cite[Lemma 5.2.3]{KP} as well as the pioneering work \cite{VK}.  For completeness however, we sketch the details here.
First, we note that  the smallest eigenvalue $\gamma_0\in\RM$ of $\mathcal{L}$ on the co-dimension two subspace $\mathcal{T}_0\cap\{\mu'\}^\perp$ can be 
determined variationally as
\begin{equation}\label{e:var1}
\gamma_0=\inf_{v\in\mathcal{T}_0,~v\perp\mu'}\frac{\left<\mathcal{L}v,v\right>}{\left<v,v\right>}
\end{equation}
while the smallest eigenvalue of $\mathcal{L}$ on $L^2(\RM)$ can be similarly determined as
\[
\lambda_0=\inf_{v\in L^2(\RM)}\frac{\left<\mathcal{L}v,v\right>}{\left<v,v\right>}.
\]
As such, we clearly have $\lambda_0\leq\gamma_0$.  Further, we know that if $\chi$ is an eigenfunction for $\lambda_0$ then we know that $\chi$ is necessarily
sign-definite.  Noting that \eqref{e:profile1} along with condition \eqref{kcond} yields the relationship
\[
\mu = \frac{a(k)}{(c-\phi^2)^{3/2}} = k\left(\frac{c-k^2}{c-\phi^2}\right)^{3/2},
\]
we further have from \eqref{e:F'} that
\[
\frac{\delta \mathcal{F}}{\delta m}(\mu) = 2\phi-2k\left(\frac{c-\phi^2}{c-k^2}\right)^{1/2}\geq 2(\phi-k)>0.
\]
It follows that $\left<\frac{\delta \mathcal{F}}{\delta m}(\mu),\chi\right>\neq 0$ and hence, by definition,  $\chi\notin\mathcal{T}_0$.
In particular, it follows that 
\begin{equation}\label{e:compare1}
\lambda_0<\gamma_0.
\end{equation}
Our next immediate goal is to show that, in fact, $\gamma_0>0$ provided that condition \eqref{e:VK} holds.

For the sake of contradiction,
assume that $\gamma_0<0$ and note that it follows by Assumption \ref{H1} and \eqref{e:compare1} that that $\gamma_0$ is necessarily not an eigenvalue of $\mathcal{L}$.
Further, by standard variational methods one can show that the infimum in \eqref{e:var1} is achieved at some $\psi\in\mathcal{T}_0$ with $\psi\perp\mu'$ that satisfies\footnote{Note
that since $\mu$ is even and $\mathcal{L}$ preserves parity, the equation below and the fact that $\frac{\delta\mathcal{F}}{\delta m}(\mu)$ is even automatically 
implies that $\psi\perp\mu'$.}
\[
\mathcal{L}\psi = \gamma_0\psi + \alpha\frac{\delta\mathcal{F}}{\delta m}(\mu)
\]
for some non-zero\footnote{Note $\alpha\neq 0$ since, otherwise, $\gamma_0$ would be an eigenvalue of $\mathcal{L}$, contradicting \eqref{e:compare1}.} constant $\alpha$. 
Since $\gamma_0$ is not an eigenvalue of $\mathcal{L}$ it follows that the above can be solved as
\[
\psi=\alpha \left(\mathcal{L}-\gamma_0\right)^{-1}\frac{\delta\mathcal{F}}{\delta m}(\mu)
\]
and that, further, the requirement $\psi\in\mathcal{T}_0$ is guaranteed by enforcing the orthogonality condition
\[
\left<\left(\mathcal{L}-\gamma_0\right)^{-1}\frac{\delta\mathcal{F}}{\delta m}(\mu),\frac{\delta\mathcal{F}}{\delta m}(\mu)\right>=0.
\]
With this motivation in mind, we define the function $g:(\lambda_0,0]\to\RM$ by
\[
g(\gamma):=\left<\left(\mathcal{L}-\gamma\right)^{-1}\frac{\delta\mathcal{F}}{\delta m}(\mu),\frac{\delta\mathcal{F}}{\delta m}(\mu)\right>
\]
and note that $g(\gamma)=0$ precisely when $\gamma$ is a negative eigenvalue of $\mathcal{L}|_{\mathcal{T}_0\cap\{\mu'\}^\perp}$.  Most notably, $\gamma_0$
must necessarily be the smallest root of $g$.  One can readily check that $g$ is a smooth function on $(\lambda_0,0)$.
Additionally, it  can  be easily shown that $g'(\gamma)>0$ for all $\gamma\in(\lambda_0,0)$ and, furthermore, that
\[
\lim_{\gamma\to\lambda_0^+}g(\gamma)=-\infty.
\]
Since we must have $g(\gamma_0)=0$, the assumption that $\gamma_0\leq 0$ implies that we must have $g(0)\geq 0$, which directly contradicts \eqref{e:VK}.  It thus
follows that $\gamma_0>0$ as claimed.

From above, it follows that the condition \eqref{e:VK} implies there exists a constant $\gamma\geq\gamma_0$ such that
\begin{equation}\label{e:bd1}
\left<\mathcal{L}v,v\right>\geq\gamma\|v\|_{L^2}^2.
\end{equation}
To upgrade this to the desired $H^1$-coercivity bound \eqref{e:L_coercive}, we use an elementary interpolation argument.  Recalling \eqref{Ldef2} and \eqref{LFG} we can write
\[
\mathcal{L}=\partial_x\left(F\partial_x\right)+G+2\omega_0\left(1-\partial_x^2\right)^{-1}
\]
where $F$ and $G$ are the smooth explicit functions given in \eqref{FGdef}.  It follows that for all $v\in\mathcal{T}_0$ with $v\perp\mu'$ we have
\begin{align*}
\left<\mathcal{L}v,v\right>&=-\int_\RM F(x)v_x^2~dx + \int_\RM G(x)v^2~dx + 2\omega_0\int_\RM v(1-\partial_x^2)^{-1}v~dx\\
&\geq \left(\inf_{x\in\RM}|F(x)|\right)\int_\RM v_x^2~dx + \left(\inf_{x\in\RM}G(x)+2\omega_0\right)\int_\RM v^2~dx,
\end{align*}
where here we have used that $F,\omega_0<0$ as well as that using \eqref{convw} and Young's convolution inequality \eqref{e:Young}.
Taken together, it follows that there exists constants $C_1>0$ and $C_2\in\RM$ such that
\begin{equation}\label{e:bd2}
\left<\mathcal{L}v,v\right>\geq C_1\int_\RM v_x^2~dx + C_2\int_\RM v^2~dx.
\end{equation}
Interpolating \eqref{e:bd1} and \eqref{e:bd2} it follows that for each $\theta\in[0,1]$ and all $v\in\mathcal{T}_0$ with $v\perp\mu'$ that
\[
\left<\mathcal{L}v,v\right>= C_1\theta\int_\RM v_x^2~dx + \left(C_2\theta+(1-\theta)\gamma\right)\int_\RM v^2~dx,
\]
which, by choosing $\theta>0$ sufficiently small that
\[
C_2\theta+(1-\theta)\gamma>0
\]
implies the $H^1$-coercivity estimate \eqref{e:L_coercive}, as desired.
\end{proof}

\begin{remark}
The condition \eqref{e:VK} guaranteeing non-non-negativity of $\mathcal{L}$ on $\mathcal{T}_0$ is oftentimes referred to as the Vakhitov-Kolokolov condition.  In 
a later section we will derive a useful analytical representation for the inner-product in \eqref{e:VK} and will provide strong numerical evidence
that the condition \eqref{e:VK} in fact holds for all smooth solitary waves $\mu(\cdot;k)$ constructed in Theorem \ref{exth}.
\end{remark}

With Lemma \ref{P:VK} in hand, we now establish that the nonlinear functional $\Lambda$ itself is coercive on the nonlinear submanifold $\mathcal{M}_k$
near the solitary wave $\mu(\cdot;k)$.  To this end, we introduce the semidistance $\rho:X_k\times X_k\to\RM$ by
\[
\rho(m_1,m_2)=\inf_{x_0\in\RM}\left\|m_1-m_2(\cdot-x_0)\right\|_{H^1}
\]
and note for a given $m\in X_k$ that $\rho(m,\mu)$ is the precisely the distance between $m$ and the group orbit $\mathcal{O}_k$ of $\mu(\cdot;k)$.  

\begin{proposition}\label{P:coercive}
Assume the hypotheses of Lemma \ref{P:VK} as well as the condition \eqref{e:VK} holds.  There exists a $\delta>0$ and a constant $C=C(\delta)>0$
such that if $m\in\mathcal{M}_k$ with $\rho(m,\mu)<\delta$ then
\begin{equation}\label{e:coercive}
\Lambda(m)-\Lambda(\mu)\geq C\rho(m,\mu)^2.
\end{equation}
\end{proposition}

\begin{proof}
To start, note by the Implicit Function Theorem that for $\delta>0$ sufficiently small there exists a unique real-valued $C^1$ map $\omega$
defined on a $\delta$-neighborhood $\mathcal{U}_\delta:=\left\{m\in H^1(\RM):\rho(m,\mu)<\delta\right\}$ of the group orbit $\mathcal{O}_k$ such that
\[
\omega(\mu)=0~~{\rm and}~~\left<m\left(\cdot+\omega(m)\right),\mu'\right>=0
\]
for all $m\in\mathcal{U}_\delta$.  Since the functional $\Lambda$ is invariant under spatial translations, it is sufficient to establish \eqref{e:coercive}
with $m$ replaced by $m(\cdot+\omega(m))$.  To this end, fix $m\in \mathcal{U}_\delta$ and note that 
\begin{equation}\label{e:decomp}
m(\cdot+\omega(m))=\mu+\alpha\frac{\delta \mathcal{F}}{\delta m}(\mu)+\eta
\end{equation}
for some constant $\alpha\in\RM$ and function $\eta\in\mathcal{T}_0$.  Note, in particular, that if $m=\mu$ then $C=y=0$.  Further,
setting $v=m(\cdot+\omega(m))-\mu$ and assuming $\|v\|_{H^1}<\delta$ we have
\[
\mathcal{F}(m)=\mathcal{F}(m(\cdot+\omega(m)))=\mathcal{F}(\mu)+\left<\frac{\delta\mathcal{F}}{\delta m}(\mu),v\right>+\mathcal{O}\left(\|v\|_{H^1}^2\right).
\]
Using \eqref{e:decomp} note that
\[
\left<\frac{\delta\mathcal{F}}{\delta m}(\mu),v\right>=C\left\|\frac{\delta\mathcal{F}}{\delta m}(\mu)\right\|_{L^2}^2
\]
and hence we have $\alpha=\mathcal{O}\left(\|v\|_{H^1}^2\right)$.  

Now, using that $\mu$ is a critical point of $\Lambda$ by Proposition \ref{e:Lmult}, we have by Taylor's theorem that
\[
\Lambda(m)=\Lambda\left(m(\cdot+\omega(m))\right)=\Lambda(\mu)+\frac{1}{2}\left<\mathcal{L}v,v\right>+o\left(\|v\|_{H^1}^2\right)
\]
and hence
\[
\Lambda(m)-\Lambda(\mu)=\frac{1}{2}\left<\mathcal{L}v,v\right>+o\left(\|v\|_{H^1}^2\right)=\frac{1}{2}\left<\mathcal{L}y,y\right>+\mathcal{O}\left(\|v\|_{H^1}^2\right).
\]
Since $y\in\mathcal{T}_0$ and $y\perp \mu'$ it follows from Lemma \ref{P:VK} that
\[
\left<\mathcal{L}y,y\right>\geq \gamma\|y\|_{H^1}^2
\]
for some constant $\gamma>0$.  Noting now that
\[
\|y\|_{H^1}\geq\left|\|v\|_{H^1}-\left\|\alpha \frac{\delta\mathcal{F}}{\delta m}(\mu)\right\|_{H^1}^2\right|\geq\|v\|_{H^1}-\widetilde{C}\|v\|_{H^1}^2
\]
for some constant $\widetilde{C}>0$, where here we again used the estimate $\alpha=\mathcal{O}\left(\|v\|_{H^1}^2\right)$, it follows there exists a constant $C>0$ such that
\[
\Lambda(m)-\Lambda(\mu)\geq C\|v\|_{H^1}^2\geq C\rho(m,\mu)^2,
\]
as desired.
\end{proof}

\subsection{Orbital Stability Result}

We can now establish our main stability result.  Before doing so, we note that in the next section we will analyze the  Vakhitov-Kolokolov condition 
\eqref{e:VK} and, in particular, in Lemma \ref{L:F_pos} that a sufficient condition for \eqref{e:VK} to hold at a solitary wave $\mu(\cdot;k_0)$ is that
\begin{equation}\label{e:F_dec}
\frac{\partial}{\partial k}\mathcal{F}(\mu(\cdot;k))<0
\end{equation}
at $k=k_0$.  Further, in the next section we will show by a straightfoward numerical calculation that condition \eqref{e:F_dec} holds for all $k\in(0,\sqrt{c}/2)$, indicating
that all smooth solitary wave solutions of the Novikov equation are orbitally stable.  Equipped with Lemma \ref{L:F_pos}, we have the following.

\begin{theorem}[Main Result] \label{T:main}
Fix $c>0$ and let $\mu(\cdot;k)$ be a smooth solitary wave solution of the Novikov equation  as  constructed in Theorem \ref{exth}
for some $k\in(0,\sqrt{c}/2)$.  Additionally, assume that   Assumption \ref{H1} and the strict inequality \eqref{e:F_dec} both hold.
Given any $\eps>0$ sufficiently small there exists a constant $C=C(\eps)>0$ such that if $v\in H^1(\RM)$ with $\|v\|_{H^1}\leq\eps$ and
if $m(\cdot,t)$ is a solution of \eqref{e:nov2} for some interval of time with initial condition $m(\cdot,0)=\mu+v$, then $m(\cdot,t)$ may
be continued to a solution for all $t>0$ such that
\[
\sup_{t>0}\inf_{x_0\in\RM}\left\|m(\cdot,t)-\mu(\cdot-x_0;k)\right\|_{H^1}\leq C\|v\|_{H^1}.
\]
\end{theorem}

\begin{proof}
Fix $k_0\in(0,\sqrt{c}/2)$ and let $\mu_0=\mu(\cdot;k_0)$ be the smooth solitary wave solution with asymptotic value $k_0$. 
Let $\delta>0$ be such that Proposition \ref{P:coercive} holds, and let $v\in H^1(\RM)$ satisfy $\rho(\mu_0+v,\mu_0)\leq\eps$ for some $0<\eps<\delta$ sufficiently
small.  By replacing $v$ by an appropriate spatial translate if necessary, we may assume that $\|v\|_{H^1}\leq \eps$.  Since $\mu$
is a critical point of $\Lambda$, Taylor's theorem implies that $\Lambda(\mu_0+v)-\Lambda(\mu)\leq C\eps^2$ for some constant $C>0$.  
Further, since $\mu_0+v\in\mathcal{M}_k$, the unique solution of \eqref{e:nov2} with initial condition $m(\cdot,0)=\mu_0+v$ must lie in
$\mathcal{M}_{k_0}$ for as long as the solution exists.  Noting that $\Lambda(m(\cdot,t))=\Lambda(\mu_0+v)$ for all $t$, it follows
by Proposition \ref{P:coercive} that $\rho(m(\cdot,t),\mu_0)\leq C\eps$ for all $t\geq 0$, as desired.

Now, suppose that $\mu_0+v\notin\mathcal{M}_{k_0}$.  In this case, we claim that we can vary the parameter $k$ slightly in order to effectively reduce this case 
to the previous one above.  Indeed, condition \eqref{e:F_dec} implies\footnote{Specifically, at this part of the argument we are only using 
that \eqref{e:F_dec} implies that $k_0$ is not a critical point of the map $k\mapsto\mathcal{F}(\mu(\cdot;k))$.} that the map
\[
k\mapsto\mathcal{F}(\mu(\cdot;k))
\]
is a diffeomorphism from a neighborhood of $k_0$ onto a neighborhood of $\mathcal{F}(\mu_0)$. In particular, we can find a constant $\Delta k$
with $|\Delta k|=\mathcal{O}(\eps)$ such that the function
\[
\widetilde{\mu}=\mu(\cdot;k_0+\Delta k)
\]
is a solution of \eqref{e:nov2} in $X_{k_0+\Delta k}$ and satisfies
\[
\mathcal{F}(\widetilde{\mu})=\mathcal{F}(\mu+v).
\]
Defining the augmented functional
\[
\widetilde{\Lambda}(m)=\omega_0(k_0+\Delta k)\mathcal{E}(m)+\omega_1(k_0+\Delta k)F_1(m)+F_2(m)
\]
on $X_{k_0+\Delta k}$, where $\omega_0$ and $\omega_1$ are defined as in \eqref{e:Lmult}, it follows as before that
\[
\widetilde{\Lambda}(m(\cdot;t))-\widetilde{\Lambda}(\widetilde{\mu})\geq C_1\rho\left(m(\cdot,t),\widetilde{\mu}\right)^2
\]
for some constant $C_1>0$ as long as $\rho\left(m(\cdot,t),\widetilde{\mu}\right)$ is sufficiently small.  Since $\widetilde{\mu}$ is a critical point
of the functional $\widetilde{\Lambda}$, we have
\[
C_1\rho\left(m(\cdot,t),\widetilde{\mu}\right)^2\leq\widetilde{\Lambda}(m(\cdot;t))-\widetilde{\Lambda}(\widetilde{\mu})\leq C_2\left\|m(\cdot,0)-\widetilde{\mu}\right\|_{H^1}^2
\]
for some constant $C_2>0$.  Moreover, by the triangle inequality we have
\[
\left\|m(\cdot,0)-\widetilde{\mu}\right\|_{H^1}\leq \left\|m(\cdot,0)-\mu\right\|_{H^1}+\left\|\mu-\widetilde{\mu}\right\|_{H^1}\leq C_3\eps 
\]
for some constant $C_3>0$ and hence there exists a constant $C_4>0$ such that
\[
\rho\left(m(\cdot,t),\widetilde{\mu}\right)\leq\rho\left(m(\cdot,0),\widetilde{\mu}\right)+\left\|\widetilde{\mu}-m(\cdot,t)\right\|_{H^1}\leq C_4\eps 
\]
for all $t>0$, completing the proof.
\end{proof}

In the above proof, we note that the argument for the case $\mu_0+v\notin\mathcal{M}_{k_0}$ is modeled after similar arguments in 
\cite{HG07,HJ15,J09} the context
of nonlinear stability of periodic traveling wave solutions in nonlinear Hamiltonian systems.

\subsection{Analysis of the Vakhitov-Kolokolov Condition}\label{S:VK}

In this section, we seek to derive a useful analytic representation for the condition \eqref{e:VK} in the previous section.  To this end, fix $c>0$, 
let $k\in(0,\sqrt{c}/2)$ and let $\mu(\cdot;k)$ be a smooth solitary wave solution from \eqref{exth}. 

We begin by observing that
the action of $\mathcal{L}$ on the various variations of the profile $\mu$ can be determined by differentiating the relation $\frac{\delta\Lambda}{\delta m}(\mu)=0$ 
with respect to the parameters $a$, $E$, and $c$, noting that the Lagrange multipliers depend explicitly on these parameters as in \eqref{e:Lmult}.  
This yields the relations
\begin{equation}\label{e:Lvar}
\left\{\begin{aligned}
\mathcal{L}\mu_a&=-\frac{\partial\omega_0}{\partial a}\frac{\partial \mathcal{E}}{\partial m}(\mu)-\frac{\partial \omega_1}{\partial a}\frac{\partial F_1}{\partial m}(\mu)\\
\mathcal{L}\mu_E&=-\frac{\partial \omega_1}{\partial E}\frac{\partial F_1}{\partial m}(\mu)\\
\mathcal{L}\mu_c&=-\frac{\partial \omega_1}{\partial c}\frac{\partial F_1}{\partial m}(\mu),
\end{aligned}\right.
\end{equation}
which will will use heavily in this section.  Of course, the variations $\mu_a$, $\mu_E$ and $\mu_c$ themselves are likely not integrable and, in fact, understanding their asymptotic
values as $x\to\pm\infty$ would require knowledge of how the asymptotic value $k$ of the profile $\mu$ varies with respect to these parameters.  Using an appropriate
scaling property, however, we identify an appropriate linear combination of these variations which is both integrable and has an easily identifiable limit at spatial infinity.  This is the
result of the following result.

\begin{lemma}\label{L:rescale}
Let $\phi$ be a bounded, smooth traveling wave solution of \eqref{e:profile_quad}, and assume that $\phi$ is smooth with respect
to the parameters $a$, $E$, and $c$.  
Then  $\mu=\phi-\phi''$ satisfies
\[
\frac{1}{2}\mu=2a\mu_a+E\mu_E+c\mu_c.
\]
\end{lemma}

\begin{proof}
Note that solutions $\phi$ of the the profile equations \eqref{e:profile1}-\eqref{e:profile_quad} satisfy the scaling symmetry
\begin{equation}\label{e:rescale1}
\phi(x;a,E,c) = c^{1/2}\psi(x;\alpha,\beta),~~a=c^2\alpha,~~E=c\beta
\end{equation}
where $\psi$ and the constants $\alpha$ and $\beta$ are independent of $c$.  Differentiating with respect to $c$ gives
\[
\frac{\partial\phi}{\partial a}\frac{\partial a}{\partial c}+\frac{\partial\phi}{\partial E}\frac{\partial E}{\partial c}+\frac{\partial\phi}{\partial c}=\frac{1}{2c^{1/2}}\psi=\frac{1}{2c}\phi.
\]
Noting that 
\[
\frac{\partial a}{\partial c}=2c\alpha=\frac{2a}{c}~~{\rm and}~~\frac{\partial E}{\partial c}=\beta=\frac{E}{c}
\]
it follows that
\[
2a\frac{\partial\phi}{\partial a}+E\frac{\partial\phi}{\partial E}+c\frac{\partial\phi}{\partial c}=\frac{1}{2}\phi.
\]
Noting that $\mu=\phi-\phi''$, it follows that the same expression holds for $\mu$, as  claimed.
\end{proof}

\begin{remark}
There is another scaling symmetry that one could use to establish the above identity.  
Indeed, note that if $\phi$ is a solution of the profile equations \eqref{e:profile1}-\eqref{e:profile_quad} then
\[
\phi(x;a,E,c) = a^{1/4}\zeta(x;\gamma,\xi),~~E=a^{1/2}\gamma,~~c=a^{1/2}\xi
\]
where here $\zeta$ and the constants $\gamma$ and $\xi$ are independent of $a$.  Differentiating the above with respect to $a$ gives
\[
\frac{\partial\phi}{\partial a} + \frac{\partial \phi}{\partial E}\frac{\partial E}{\partial a} + \frac{\partial\phi}{\partial c}\frac{\partial c}{\partial a}=\frac{1}{4a^{3/4}}\zeta
\]
Noting that
\[
\frac{\partial E}{\partial a}=\frac{\gamma}{2a^{1/2}} = \frac{E}{2a}~~{\rm and}~~\frac{\partial c}{\partial a}=\frac{\xi}{2a^{1/2}}=\frac{c}{2a}
\]
it follows that
\[
2a\frac{\partial\phi}{\partial a} + E\frac{\partial \phi}{\partial E}+ c\frac{\partial\phi}{\partial c}=\frac{a^{1/4}}{2}\zeta=\frac{1}{2}\phi.
\]
Again, noting that $\mu=\phi-\phi''$ will satisfy precisely the same identify this leads to the result of Lemma \ref{L:rescale}.  This shows it doesn't matter which scaling symmetry you exploit, they lead to the same result.
\end{remark}

Using the identities in \eqref{e:Lvar} along with Lemma \ref{L:rescale}, it follows that 
\[
\mathcal{L}\mu=-4a\frac{\partial\omega_0}{\partial a}\frac{\partial \mathcal{E}}{\partial m}(\mu)
	-2\left(2a\frac{\partial \omega_1}{\partial a}+E\frac{\partial \omega_1}{\partial E}+c\frac{\partial \omega_1}{\partial c}\right)\frac{\partial F_1}{\partial m}(\mu).
\]
Furthermore, using now \eqref{kcond} to express the Lagrange multipliers $\omega_0$ and $\omega_1$ as explicit functions of $k$, we have by differentiating
the equation $\frac{\delta\Lambda}{\delta m}(\mu)=0$ with respect to $k$ gives
\[
\mathcal{L}\mu_k=-\frac{\partial\omega_0}{\partial k}\frac{\partial \mathcal{E}}{\partial m}(\mu)-\frac{\partial \omega_1}{\partial k}\frac{\partial F_1}{\partial m}(\mu),
\]
By construction then, we note that
\[
\mu(x)\to k~~{\rm and}~~\mu_k(x)\to 1
\]
as $x\to\pm\infty$ and hence, in particular, that the function $k\mu_k-\mu$ belongs to $H^1(\RM)$ and
\begin{align*}
\mathcal{L}\left(k\mu_k-\mu\right) &= -\left(k\frac{\partial\omega_0}{\partial k}-4a\frac{\partial\omega_0}{\partial a}\right)\frac{\partial\mathcal{E}}{\partial m}(\mu)\\
	&\qquad-\left(k\frac{\partial \omega_1}{\partial k}-4a\frac{\partial \omega_1}{\partial a}-2E\frac{\partial \omega_1}{\partial E}-2c\frac{\partial \omega_1}{\partial c}\right)
			\frac{\partial F_1}{\partial m}(\mu).
\end{align*}
Computing the various derivatives of the Lagrange multipliers, and using \eqref{kcond} to substitute explicit expressions for $a=a(k)$, $E=E(k)$ everywhere, we find that 
\begin{align*}
\mathcal{L}\left(k\mu_k-\mu\right)  &=\left(\frac{18c}{k^{2/3}(c-k^2)^2}\right)\frac{\partial \mathcal{E}}{\partial m}(\mu)
		-\left(\frac{54ck^{2/3}}{(c-k^2)^2}\right)\frac{\partial F_1}{\partial m}\\
&=\frac{18c}{k^{2/3}(c-k^2)^2}\left(\frac{\partial\mathcal{E}}{\partial m}(\mu)-3k^{4/3}\frac{\partial F_1}{\partial m}(\mu)\right)\\
&=\frac{18c}{k^{2/3}(c-k^2)^2}\frac{\delta\mathcal{F}}{\delta m}(\mu),
\end{align*}
where the last equality follows by definition from \eqref{e:F}.   Taken together, it follows that
\[
\mathcal{L}^{-1}\frac{\delta\mathcal{F}}{\delta m}(\mu)=\frac{k^{2/3}(c-k^2)^2}{18c}\left(k\mu_k-\mu\right)
\]
and hence that the condition \eqref{e:VK} can be expressed as
\[
\left<\mathcal{L}^{-1}\frac{\delta\mathcal{F}}{\delta m}(\mu),\frac{\delta\mathcal{F}}{\delta m}(\mu)\right>
  = \frac{k^{2/3}(c-k^2)^2}{18c}\int_\RM \frac{\delta\mathcal{F}}{\delta m}(\mu)\left(k\mu_k-\mu\right)dx.
\]
With this representation, we now establish the following.

\begin{lemma}
Under the hypotheses of Lemma \ref{P:VK} we have
\[
\left<\mathcal{L}^{-1}\eta_0,\eta_0\right> = \frac{k^{11/3}(c-k^2)^2}{18c}\frac{\partial}{\partial k}\left[ \frac{1}{k^2}\mathcal{F}(\mu(\cdot;k))\right].
\]
\end{lemma}

\begin{proof}
Noting that the functional $\mathcal{F}$ depends explicitly on $k$, we have
\[
\frac{\partial}{\partial k}\left(\mathcal{F}(\mu(\cdot;k))\right)=\int_\RM\frac{\delta\mathcal{F}}{\delta m}(\mu)\mu_k~dx-4k^{1/3}F_1(\mu)
\]
and hence
\[
\int_\RM\frac{\delta\mathcal{F}}{\delta m}(\mu)\left(k\mu_k-\mu\right)dx = k\frac{\partial}{\partial k}\left(\mathcal{F}(\mu(\cdot;k))\right)+4k^{4/3}F_1(\mu)
	-\int_\RM\frac{\delta\mathcal{F}}{\delta m}(\mu)\mu~dx.
\]
Further, using the explicit formulas for $\mathcal{E}$ and $F_1$ we find that
\begin{align*}
\int_\RM\frac{\delta\mathcal{F}}{\delta m}(\mu)\mu~dx &= \int_\RM\left(\frac{\delta\mathcal{E}}{\delta m}(\mu)-3k^{4/3}\frac{\delta F_1}{\delta m}(\mu)\right)\mu~dx\\
&=\int_\RM\left(2(1-\partial_x^2)^{-1}\mu-2k^{4/3}\mu^{-1/3}\right)\mu~dx\\
&=\int_\RM\left[2\left(\mu(1-\partial_x^2)^{-1}\mu-k^2\right)-2k^{4/3}\left(\mu^{2/3}-k^{2/3}\right)\right]~dx\\
&=2\mathcal{E}(\mu)-2k^{4/3}F_1(\mu),
\end{align*}
where in the third equality we are subtracting off the asymptotic end states so the integral can be split into two pieces\footnote{Thankfully, these end state contributions cancel perfectly.}.
It thus follows that
\begin{align*}
\int_\RM\frac{\delta\mathcal{F}}{\delta m}(\mu)\left(k\mu_k-\mu\right)dx &= k\frac{\partial}{\partial k}\left(\mathcal{F}(\mu(\cdot;k))\right)
		-2\mathcal{E}(\mu)+6k^{4/3}F_1(\mu)\\
&=k\frac{\partial}{\partial k}\left(\mathcal{F}(\mu(\cdot;k))\right)-2\mathcal{F}(\mu)\\
&=k^3\frac{\partial}{\partial k}\left(k^{-2}\mathcal{F}(\mu(\cdot;k))\right),
\end{align*}
which completes the proof.
\end{proof}

To aid in calculating the above derivaitve, we have the following result.

\begin{lemma}\label{L:F_pos}
The function $\mathcal{F}(\mu(\cdot,k))$ is strictly positive for $k\in(0,\sqrt{c}/2)$.  In particular, a sufficient condition for \eqref{e:VK}
to hold is that 
\[
\frac{\partial}{\partial k}\mathcal{F}\left(\mu(\cdot,k)\right)<0.
\]
\end{lemma}

\begin{proof}
Recalling the definition of $\mathcal{F}$ in \eqref{e:F} we note that
\begin{align*}
\mathcal{F}(\mu(\cdot,k)) &= \int_\RM\left(\mu\phi-3k^{4/3}\mu^{2/3}+2k^2\right)dx\\
&=\int_\RM\left[k\phi\left(\frac{c-k^2}{c-\phi^2}\right)^{3/2}-3k^2\left(\frac{c-k^2}{c-\phi^2}\right)+2k^2\right]dx,
\end{align*}
where for the second equality we have used  \eqref{e:profile1} along with the condition \eqref{kcond} to determine the relationship
\[
\mu = \frac{a(k)}{(c-\phi^2)^{3/2}} = k\left(\frac{c-k^2}{c-\phi^2}\right)^{3/2}.
\]
It is sufficient to prove that the integrand above is positive which, after rearranging is equivalent to showing that
\[
\phi(x)\sqrt{\frac{c-k^2}{c-\phi(x)^2}}-3k+2k\left(\frac{1-\phi(x)^2}{1-k^2}\right)>0
\]
for all $x\in\RM$.  To this end, we define the function
\[
f(z)=x\sqrt{\frac{c-k^2}{c-z}}-3k+2k\left(\frac{1-z^2}{1-k^2}\right)
\]
and note it is sufficient to prove that $f$ is positive for $z\in[k,\sqrt{c})$.  A quick calculation shows $f(k)=0$ and, further, 
\[
f'(z) = \left(\frac{c}{c-z^2}\right)\sqrt{\frac{c-k^2}{c-z^2}}-\frac{4kz}{c-k^2},
\]
which is postiive provided that
\[
c(c-k^2)^3>16k^2z^2(c-z^2)^3=:g(z)
\]
for all $z\in[k,\sqrt{c})$.  Note that  $g$ is maximized on $[k,\sqrt{c})$ at $z=\sqrt{c}/2$, at which point
\[
g\left(\frac{\sqrt{c}}{2}\right) = \frac{27c^4}{256}.
\]
Finally, since \eqref{kcond2} implies that
\[
c(c-k^2)^3>\frac{27c^4}{64},
\]
it follows that $f'(z)>0$ for all $z\in[k,\sqrt{c})$.  Since $f(k)=0$, as already mentioned, it follows that $f(z)>0$ for all $z\in[k,\sqrt{c}))$,
as desired.  

With the positivity of $\mathcal{F}(\mu(\cdot;k))$ being established, the last claim follows by simply noting that $k$ satisfies \eqref{kcond2} and
\[
\frac{\partial}{\partial k}\left(\frac{1}{k^2}\mathcal{F}(\mu(\cdot,k))\right) = \frac{1}{k^2}\frac{\partial}{\partial k}\left(\mathcal{F}(\mu(\cdot,k))\right) -\frac{2}{k^3}\mathcal{F}(\mu(\cdot,k)),
\]
and hence \eqref{e:VK} holds provided that $\mathcal{F}(\mu(\cdot,k))$ is a strictly decreasing function of $k$, as claimed.
\end{proof}

%

\begin{figure}[t!]
\begin{center}
\includegraphics[scale=0.5]{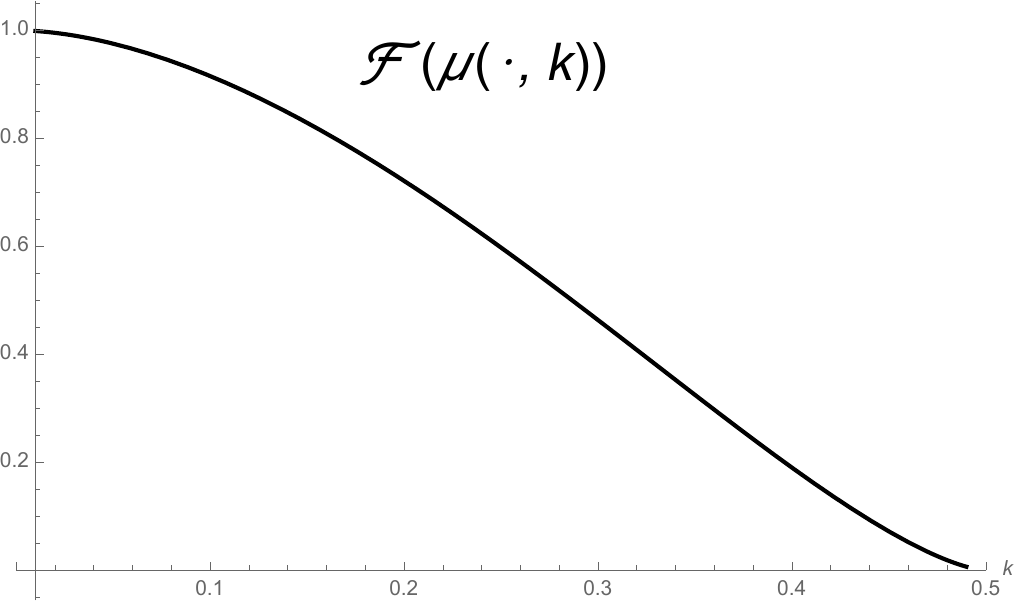}
\caption{A plot of $\mathcal{F}(\mu(\cdot;k))$ vs. $k\in(0,\sqrt{c}/2)$ for $c=1$.  Note, in particular, that it follows that \eqref{e:F_dec} holds for all $k\in(0,\sqrt{c}/2)$.}
\label{F:F_dec}
\end{center}
\end{figure}

It now remains to calculate the derivative \eqref{e:F_dec}.  The strategy here is to use the existence theory from Section \ref{S:exist}
to express the function $\mathcal{F}$ in terms of quadrature.  To this end, 
fix $c>0$ and recall that Theorem \ref{exth} provides a family of even solitary waves $\phi(x;k)$ that satisfy
\[
\frac{1}{2}\left(\frac{\partial\phi}{\partial x}\right)^2=E(k)-V(\phi;a(k),c),
\]
where here $E(k)$ and $a(k)$ are defined explicitly in \eqref{kcond}.  Further, from \eqref{e:profile1} we have the relation
\[
\mu=(1-\partial_x^2)\phi=\frac{a(k)}{(c-\phi^2)^{3/2}}
\]
which allows us to express the momentum density of our profile $\phi$ in terms of $\phi$ itself.  Recalling from Theorem \ref{exth} that
$\phi$ is even and strictly increasing on $(-\infty,0]$ it follows that
\begin{align*}
\mathcal{E}(\mu(\cdot;k))&=\int_\RM\left(\mu \phi -k^2\right)dx\\
	&=\int_\RM\left(\frac{a(k)\phi}{(c-\phi^2)^{3/2}}-k^2\right)dx\\
	&=2\int_{k}^{\phi_M(k)}\left(\frac{a(k)\phi}{(c-\phi^2)^{3/2}}-k^2\right)\frac{d\phi}{\sqrt{2\left(E(k)-V(\phi;a(k),c)\right)}}
\end{align*}
where here $\phi_M(k)$ is the global max of the profile $\phi$.  Similarly, we have
\begin{align*}
F_1(\mu(\cdot;k))&=\int_\RM\left(\mu^{2/3}-k^{2/3}\right)dx\\
	&=\int_\RM\left[\left(\frac{a(k)}{(c-\phi^2)^{3/2}}\right)^{2/3}-k^{2/3}\right]dx\\
	&=2\int_{k}^{\phi_M(k)}\left(\frac{a(k)^{2/3}}{c-\phi^2}-k^{2/3}\right)\frac{d\phi}{\sqrt{2\left(E(k)-V(\phi;a(k),c)\right)}}.
\end{align*}

Putting the above expressions together yields the representation
\begin{equation}\label{e:Fnew}
\begin{aligned}
\mathcal{F}(\mu(\cdot;k))&=\mathcal{E}(\mu(\cdot;k))-3k^{4/3}F_1(\mu(\cdot;k))\\
	&=2\int_{k}^{\phi_M(k)}\left[\frac{a(k)\phi}{(c-\phi^2)^{3/2}}-3k^{4/3}\frac{a(k)^{2/3}}{c-\phi^2}+2k^2\right]\frac{d\phi}{\sqrt{2\left(E(k)-V(\phi;a(k),c)\right)}}.
\end{aligned}
\end{equation}
By noting that $c>0$ can be rescaled to $c=1$ using \eqref{e:rescale1}, the dependence of $\mathcal{F}(\mu(\cdot;k))$ on $k\in[0,1/2)$ is depicted in Figure \ref{F:F_dec}.
Fixing $c=1$, the maximum $\phi_M(k)$  was be computed from \eqref{e:profile_quad} for each $k\in(0,1/2)$ via a standard Newton iteration, after which
the integral in \eqref{e:Fnew} was numerically computed using standard computational software\footnote{In this work, we used Mathematica.}.
From Figure \ref{F:F_dec} we see that the Vakhitov-Kolokolov condition \eqref{e:F_dec} holds for all $k\in(0,\sqrt{c}/2)$, indicating the orbital stability
of all smooth solitary wave solutions of the Novikov equation.

\bibliographystyle{unsrt}

\end{document}